\documentclass[a4paper,reqno]{amsart}
\usepackage[english]{babel}
\usepackage{a4wide}
\usepackage{pinlabel-test}
\usepackage{graphicx}
\usepackage{amsmath,amssymb,amsthm, amsgen,mathrsfs}
\usepackage[latin1]{inputenc}
\usepackage{listings}
\usepackage{color}
\usepackage[usenames,dvipsnames]{xcolor}
\usepackage{rotating}
\usepackage{hhline}
\usepackage{subfigure}
\usepackage{multirow}
\usepackage{enumerate}
\usepackage[bookmarks]{}
\usepackage{amsfonts}   
\usepackage{dsfont}
\usepackage{tikz}
\usepackage{url}
\usepackage{booktabs}
\usepackage{tensor}

\numberwithin{equation}{section}
\newtheorem{thm}{Theorem}[section]

\newtheorem{lem}[thm]{Lemma}

\theoremstyle{definition}
\newtheorem{rem}[thm]{Remark}

\def\weakto{\rightharpoonup}

\newcommand{\R}{\mathbb R}
\newcommand{\N}{\mathbb N}

\newcommand{\C}{\mathbb C}
\newcommand{\e}{\varepsilon}
\newcommand{\bighaa}[1]{\big( #1 \big)}
\newcommand{\bin}[2]{ \left( \hspace{-1mm}
                        \begin{array}{c}
                          #1 \\
                          #2 \\
                        \end{array}
                      \hspace{-1mm} \right) }

\newcommand{\rhomin}{ \rho_\ast }
\newcommand{\tilderhomin}{\tilde\rho_\ast}

\def\rscld#1{\gamma_\to #1}
\def\antirscld#1{\gamma_\leftarrow #1}

\DeclareMathOperator\supp{supp}

\DeclareFontFamily{U}{mathx}{\hyphenchar\font45}
\DeclareFontShape{U}{mathx}{m}{n}{
      <5> <6> <7> <8> <9> <10>
      <10.95> <12> <14.4> <17.28> <20.74> <24.88>
      mathx10
      }{}
\DeclareSymbolFont{mathx}{U}{mathx}{m}{n}
\DeclareFontSubstitution{U}{mathx}{m}{n}
\DeclareMathAccent{\widecheck}{0}{mathx}{"71}

\begin{document}
\title[Boundary layers in a pile-up of dislocations walls]{Boundary-layer analysis of a pile-up of walls of edge dislocations at a lock}

\author{Adriana Garroni \and Patrick van Meurs \and Mark Peletier \and Lucia Scardia}

\address[A. Garroni]{Dipartimento di Matematica, Sapienza Universit\`a di Roma, Roma, Italy}
\email{garroni@mat.uniroma1.it}

\address[P. van Meurs]{Institute of Science and Engineering, Kanazawa University, Japan}
\email{pmeurs@staff.kanazawa-u.ac.jp}

\address[M. Peletier]{Department of Mathematics and Computer Science and Institute for Complex Molecular Systems, TU Eindhoven, The Netherlands}
\email{M.A.Peletier@tue.nl}

\address[L. Scardia]{Department of Mathematical Sciences, University of Bath, United Kingdom}
\email{L.Scardia@bath.ac.uk}

\begin{abstract}
In this paper we analyse the behaviour of a pile-up of vertically periodic walls of edge dislocations at an obstacle, represented by a locked dislocation wall. Starting from a continuum non-local energy $E_\gamma$
modelling the interactions---at a typical length-scale of $1/\gamma$---of the walls subjected to a constant shear stress, we derive a first-order approximation of the energy $E_\gamma$ in powers of $1/\gamma$ by $\Gamma$-convergence, in the limit $\gamma\to\infty$. 
While the zero-order term in the expansion, the $\Gamma$-limit of $E_\gamma$, captures the `bulk' profile of the density of dislocation walls in the pile-up domain, 
the first-order term in the expansion is a `boundary-layer' energy that captures the profile of the density in the proximity of the lock.

This study is a first step towards a rigorous understanding of the behaviour of dislocations at obstacles, defects, and grain boundaries.
\end{abstract}

\maketitle

\section{Introduction and motivation}

Dislocations are defects in the arrangement of the atoms in a metal lattice, and their distribution and motion greatly affect the macroscopic behaviour of the material. 
Dislocations do not only interact with other dislocations, but also with impurities, other defects, and interfaces. Understanding the behaviour of dislocations 
at grain boundaries and phase boundaries, in particular, has been the object of intensive research in academia and industry, but is still far from being achieved. 

The analysis of idealised pile-ups of two-dimensional point dislocations, representing the intersections of a system of straight and parallel dislocations with a cross-sectional plane, is a first attempt to shed light on this complex subject. The simplest case concerns a one-dimensional array of points, corresponding to dislocations in a single glide plane, forced against an obstacle by an external load. For this setup Eshelby, Frank and Nabarro \cite{EshelbyFrankNabarro51} pioneered a method for finding the equilibrium density of the dislocations, based on a polynomial representation, which is made possible by the specific structure of the interaction potential. This method has since been extended in various ways~\cite{VoskoboinikovChapmanMcLeodOckendon09, VoskoboinikovChapmanOckendonAllwright07, VoskoboinikovChapmanOckendon07}. Other approaches to the upscaling of dislocations are obtained via $\Gamma$-convergence  for  phase-field energies~\cite{GarroniMueller05, FocardiGarroni07, ContiGarroniMueller11}, for core-radius regularized models \cite{GarroniLeoniPonsiglione10, AlicandroDeLucaGarroniPonsiglione14}, and via homogenization for evolution equations~\cite{ElHajjIbrahimMonneau09, ForcadelImbertMonneau12, MonneauPatrizi12}.

In this paper we study a model that lies half-way between one and two dimensions. At the microscopic level, the system contains a \emph{large number of periodic walls of edge dislocations} with the same Burgers vector (see Figure~\ref{fig:setup:walls} and~\cite{RoyPeerlingsGeersKasyanyuk08, GeersPeerlingsPeletierScardia13, ScardiaPeerlingsPeletierGeers14, VanMeursMunteanPeletier14, Hall11}). The coordinate system is chosen so that the walls are vertical, and represented by their horizontal positions $\tilde x^n=(\tilde x^n_i)_{i=0}^n\in [0,\infty)^{n+1}$, with $\tilde x^n_0=0$, $n\in \mathbb{N}$. The energy of this system is given by
\begin{equation}
\label{discreteEcn-tilde}
\tilde E_n (\tilde x^n) := K_n \sum_{k=1}^n \sum_{j = 0}^{n - k} V \left(\frac{\tilde x_{j+k}^n - \tilde x_j^n}{h_n}\right) + \sigma_n\sum_{i = 1}^{n} \tilde x_i^n,
\end{equation}
where $K_n$ measures the elastic properties of the medium, $h_n$ is the distance between consecutive dislocations within a wall, $\sigma_n$ is an imposed shear stress, and the interaction energy potential $V$ (see Figure~\ref{PlotV}) is 
\begin{equation}
V(s):= s\coth s - \log (2\sinh s)\label{defV}
= \frac{2|s|}{(e^{2|s|}-1)} - \log(1-e^{-2 |s|}).
\end{equation}
The physical quantities $K_n$, $h_n$ and $\sigma_n$ are assumed to depend on $n$ since we are interested in the limit behaviour of the system for $n\to \infty$, and this depends strongly on the behaviour of all these parameters.

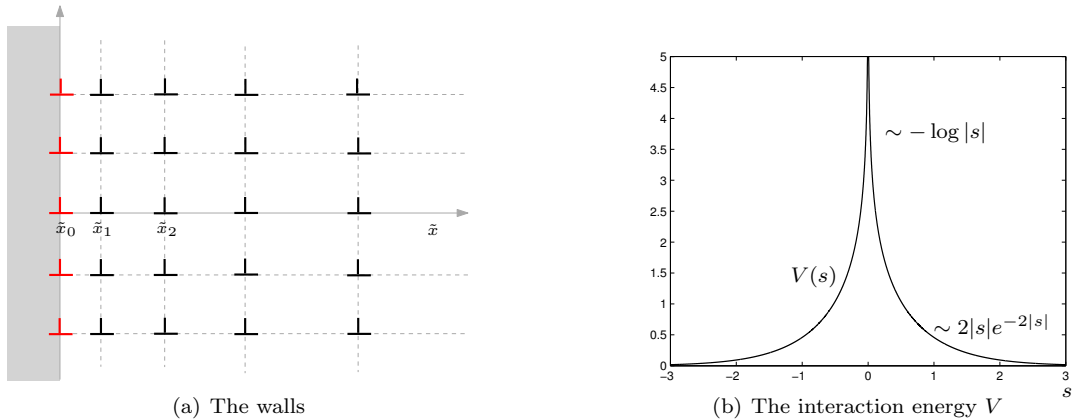
\begin{figure}[h!]
\subfigure[The walls]{
\labellist
\tiny
\pinlabel $\tilde x_0$ at 90 208
\pinlabel $\tilde x_1$ at 124 208
\pinlabel $\tilde x_2$ at 184 208
\pinlabel $\tilde x$ at 431 208
\footnotesize
\endlabellist
\includegraphics[width=2.4in]{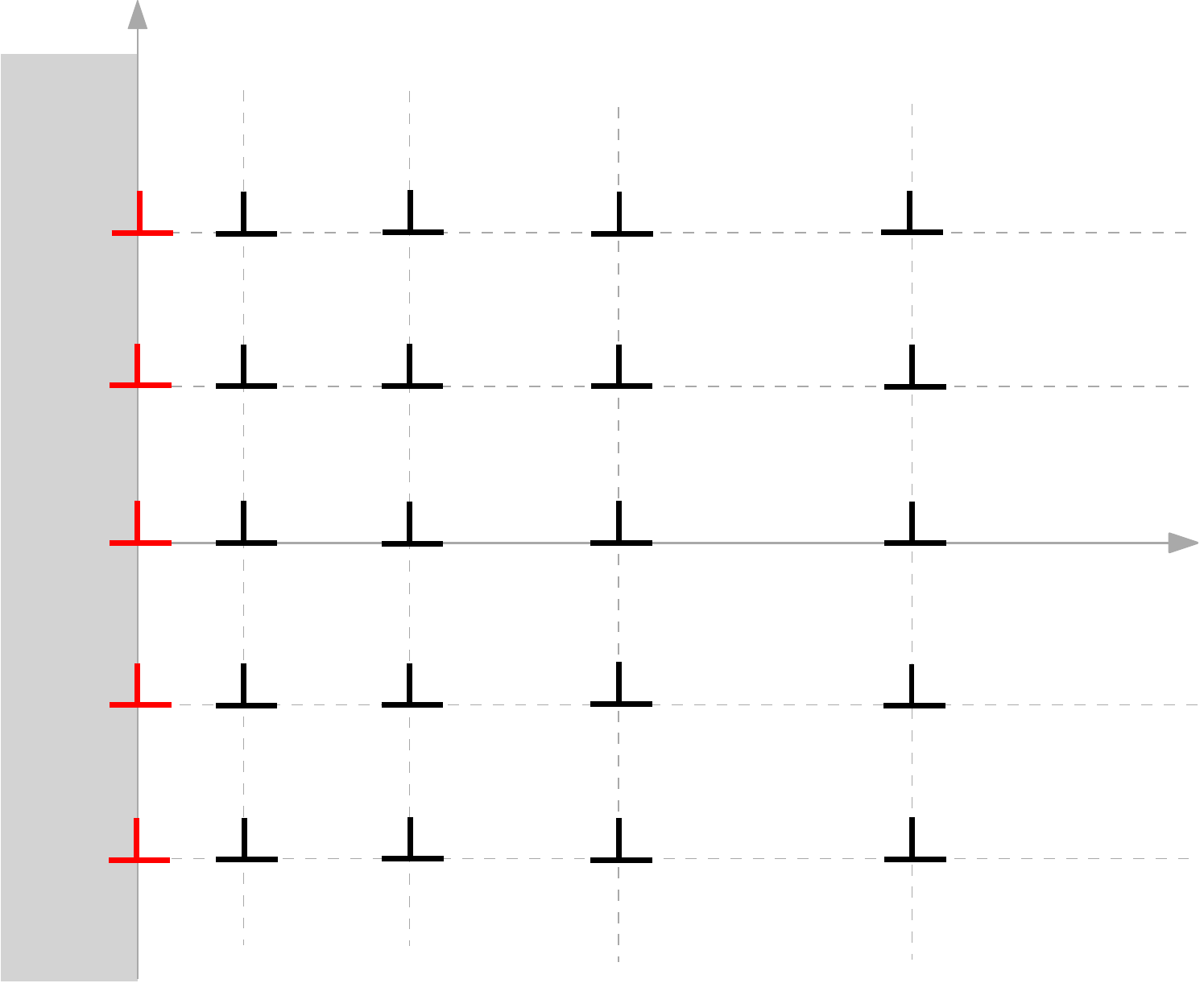}
\label{fig:setup:walls}
}
\hskip2cm
\subfigure[{The interaction energy  $V$ }]{
\labellist
\footnotesize
\pinlabel $V(s)$ [tr] at 160 100
\pinlabel $s$ at 345 -10
\pinlabel $\sim -\log|s|$ [tl] at 190 220
\pinlabel $\sim2|s|e^{-2|s|}$ [bl] at 230 30
\endlabellist
\includegraphics[width=2.2in]{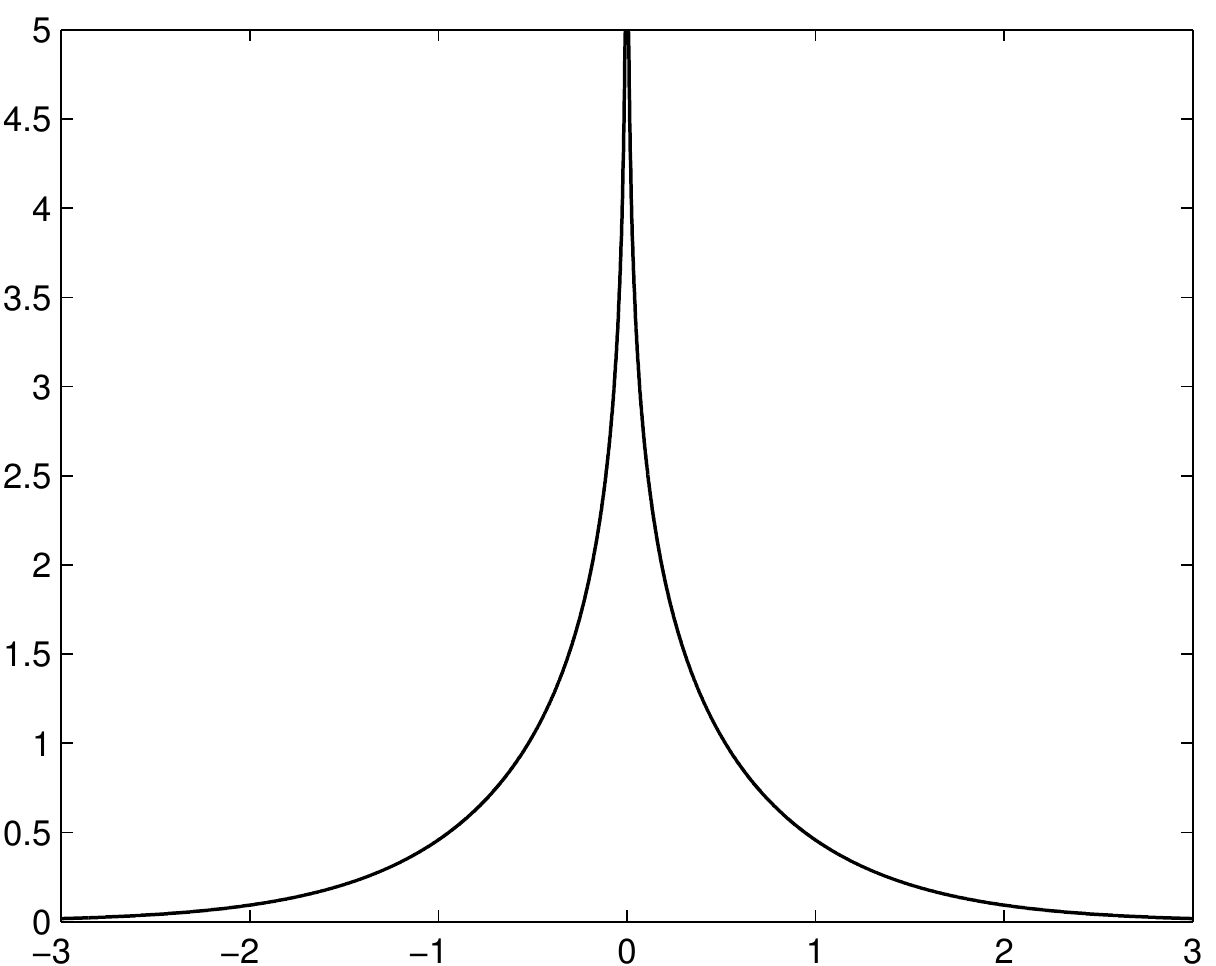}
\label{PlotV}
}
\caption{The dislocation configuration considered in this paper. Infinite, vertical \emph{walls} of equispaced dislocations are free to move in the horizontal direction. A wall is pinned at $\tilde x_0=0$ and acts as a repellent.}
\end{figure}

The interaction potential $V$ is obtained by adding up the dislocation-dislocation interaction potentials (derived from the Peach-K\"ohler force) of all the individual dislocations in a wall \cite[(19--75)]{HirthLothe82}. Note that, although the interaction between two dislocations can be attractive or repulsive, depending on their relative position, the wall-wall interaction is purely repulsive. For this reason, the first term of the energy \eqref{discreteEcn-tilde} favours configurations where dislocation walls are infinitely far from one another. The second term in \eqref{discreteEcn-tilde}, however, is a confinement potential, and models a constant shear stress acting on the system, which forces the walls against the obstacle at $\tilde x=0$. This obstacle is a locked dislocation wall.

\smallskip

As shown in \cite{GeersPeerlingsPeletierScardia13}, the problem can be rescaled to depend only on the single dimensionless parameter $\gamma_n$ \footnote{Our dimensionless parameter $\gamma_n$ corresponds to $n \beta_n$ in ~\cite{GeersPeerlingsPeletierScardia13}.}
\begin{equation*} 
  \gamma_n := \sqrt{ \frac{ n K_n }{ \sigma_n h_n } }.
\end{equation*}
Roughly speaking, $\gamma_n$ is a measure of the total length of the pile-up, relative to $h_n$. We will restrict our attention to the case when $1\ll \gamma_n\ll n$, which corresponds to arrangements where dislocations are closer horizontally than vertically, while the length of the pile-up region is larger than the in-wall spacing $h_n$. 

In term of the rescaled positions $x_i = \frac{\tilde x_i}{\gamma_n h_n}$ the (suitably rescaled) energy \eqref{discreteEcn-tilde} becomes 

\begin{equation}
\label{discreteEcn}
E_n (x^n) := \frac{\gamma_n}{n^2} \sum_{k=1}^n \sum_{j = 0}^{n - k} V \bighaa{ \gamma_n \bighaa{ x_{j+k}^n - x_j^n } } + \frac1n \sum_{i = 1}^{n} x_i^n.
\end{equation}

\smallskip

In \cite{GeersPeerlingsPeletierScardia13} upscaled continuum models were derived from $E_n$, by $\Gamma$-convergence, in the ``many-walls" limit $n\to \infty$, for different 
asymptotic behaviours of the parameter $\gamma_n$. In particular, for the scaling regime $1\ll \gamma_n \ll n$, the $\Gamma$-limit of $E_n$ was proved to be the continuum energy $E$ given by
\begin{equation}\label{limE3}
E(\mu) =\frac12 \Bigl(\int_\R V\Bigr) \int_0^\infty \rho(x)^2\, dx+ \int_0^\infty x\rho(x)\, dx, \quad \textrm{if}\ \ \mu(dx) = \rho(x)\, dx,\quad \supp \rho\subset [0,\infty),
\end{equation}
where $\mu_n:=\frac1n\sum_i\delta_{x_i^n}\weakto \mu$ as $n\to\infty$.

Figure \ref{DC3} shows a comparison between the minimiser of the discrete energy $E_n$, for $\gamma_n =\sqrt n$ and for large~$n$, and the minimiser $\rhomin$ of the continuum energy $E$. Note that the continuous minimiser $\rhomin$ is an affine function with slope $-(\int_\R V)^{-1}$ (see Remark \ref{rem:minz:E3}). In this figure we compare the two minimisers by plotting the corresponding densities; the \emph{discrete density} $\rho_n$ is defined in terms of the minimiser $x_{\ast}^n$ of $E_n$ as 
\begin{equation}\label{intro:rhon}
  \rho_n (x^n_{\ast, i}) := \frac{2/n}{ x^n_{\ast, i+1} - x^n_{\ast, i-1} },
  \qquad
  i = 1,\ldots,n-1.
\end{equation}

\begin{figure}[h]
\labellist
\pinlabel {\small Dislocation-wall position} [b] at 185 -27
\pinlabel \small $\rhomin$ at 337 230
\pinlabel \small $\rho_n$ at 337 250
\endlabellist
\begin{center}
\includegraphics[width=3in]{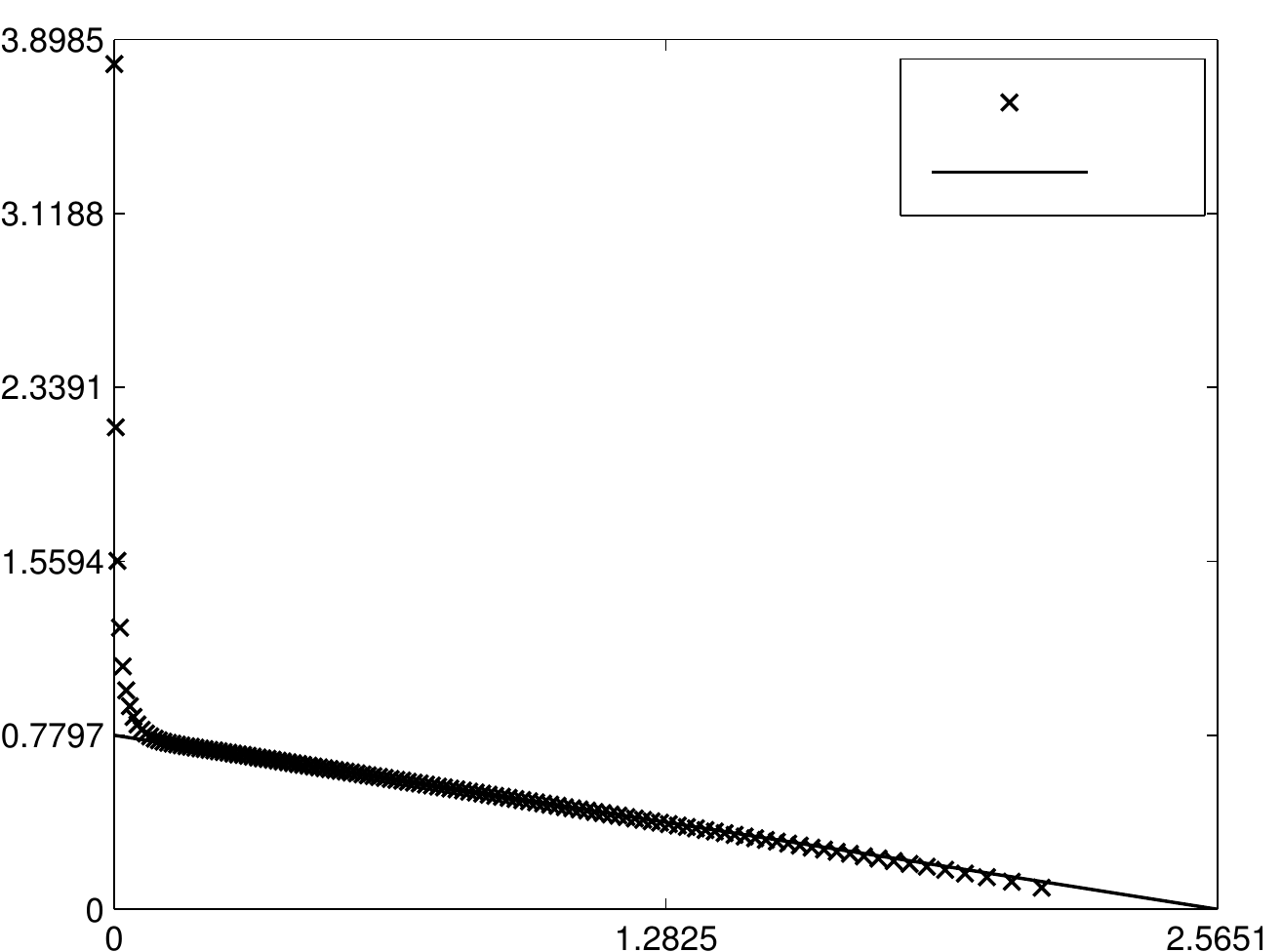}
\end{center}
\caption{Minimisers of $E_n$ and $E$, for $n=2^7$ and $\gamma_n = \sqrt n$. }
\label{DC3}
\end{figure}

\medskip

The figure illustrates well the starting point of this paper. The upscaled continuum model ($E$) fits very well with the discrete model ($E_n$) in the \textit{bulk} of the pile-up region; however, it fails to capture the distribution of dislocations at the two ends of the domain, where \textit{boundary layers} appear. In particular, the continuum model underestimates the dislocation density at the lock, and understanding the optimal arrangement of dislocations at the obstacle is the fundamental goal of the pile-up analysis.

\medskip

Inspired by this observation, the goal of this paper is to analyse the boundary layer at the lock at $x=0$, in the scaling regime $1\ll \gamma_n \ll n$. We do this by studying a  $\Gamma$-expansion of the energy $E_n$~\cite{AnzellottiBaldo93,BraidesCicalese07,BraidesTruskinovsky08} in terms of the small parameter $1/\gamma_n$.  The zero-order term of the expansion is the $\Gamma$-limit of the energy (namely $E$ in \eqref{limE3}), which describes correctly the bulk behaviour of the minimiser; the term of order $1/\gamma_n$ in the expansion, instead, is a first-order correction that captures boundary-layer effects. 

In order to capture the main features of the asymptotic expansion by $\Gamma$-convergence, and isolate them from the more technical issues related with the passage from discrete to continuum, we study the $\Gamma$-expansion of a \emph{continuum version} of $E_n$. To motivate this continuum version, note that  the discrete energies $E_n$ can formally be written in terms of discrete integrals with respect to the measures $\mu_n=\frac1n\sum_{i=1}^n\delta_{x^n_i}$, as\footnote{This expression is formal since $V(0)=+\infty$; to make it rigorous we should remove the diagonal $i=j$ from the product measure $\mu_n\otimes\mu_n= n^{-2}\sum_{i,j} \delta_{(x_i^n,x_j^n)}$.} 
\begin{equation*} 
E_n (x^n) = \frac12 \gamma_n \int_0^\infty \int_0^\infty V (\gamma_n(x-y))\mu_n(dx)\mu_n(dy) + \int_0^\infty x\mu_n(dx).
\end{equation*}
Inspired by this we define the continuum energy 
\begin{equation}\label{contEc}
E_\gamma (\mu):= \frac12\gamma \int_0^\infty \int_0^\infty V (\gamma(x-y))\mu(dx)\mu(dy) + \int_0^\infty x\mu(dx),
\end{equation}
which now is considered to be defined for all $\mu\in \mathcal{P}(0,\infty)$ -- not only for sums of Diracs. This is the functional  that we will study in this paper, in the limit $\gamma\to \infty$.

\subsection{Approximation of $E_\gamma$ by $\Gamma$-expansion} The discrete energies $E_n$ in \eqref{discreteEcn} and their continuous counterpart $E_\gamma$ in \eqref{contEc} both have as  $\Gamma$-limit the energy $E$ in \eqref{limE3}, as we show in Theorem~\ref{thm:Gamma:conv:E2c}. The limit energy $E$ is therefore the first term of the approximation of $E_\gamma$ by $\Gamma$-convergence, in the limit $\gamma\to\infty$.  Following \cite{AnzellottiBaldo93}, the next term of the expansion would be the $\Gamma$-limit of the rescaled energies 
\begin{equation}\label{firstorder-classical}
\frac{E_\gamma(\mu) - \min E}{1/\gamma} = \gamma\left(E_\gamma(\mu) - E(\rhomin)\right),
\end{equation} 
as $\gamma\to \infty$. These energies, however, do not seem appropriate for isolating the boundary-layer contribution at the lock $x=0$. Indeed, it is possible to show that there are two order-one contributions in \eqref{firstorder-classical}: one is due to the presence of the boundary layer (namely to the difference between $\rhomin$ and the minimiser of $E_\gamma$), 
and the other one to the approximation of $E_\gamma$ with $E$. The latter is actually a constant with respect to the variable $\mu$ (namely $\gamma(E_\gamma(\rhomin)-E(\rhomin))$) and does not affect the behaviour of the minimisers. Hence, to focus on the boundary-layer energy contribution only, we consider the functional 
\begin{equation}\label{intro:firstorder}
\gamma\left(E_\gamma(\mu) - E_\gamma(\rhomin)\right).
\end{equation} 
The energy in \eqref{intro:firstorder} has the advantage of {measuring} the difference between the minimiser of $E_\gamma$, which exhibits boundary layers, and its approximate bulk behaviour $\rhomin$, in terms of the \emph{same} energy $E_\gamma$ (rather than two different energies, as in \eqref{firstorder-classical}). For this reason it is genuinely a `boundary-layer energy', since the only contribution of order one in \eqref{intro:firstorder} is due to the different behaviour of the minimiser of $E_\gamma$ and $\rho_\ast$ close to $x=0$. 

\medskip

In order to capture the boundary-layer behaviour, we need to blow up the region next to $x=0$. The scaling factor for the blow up arises naturally from \eqref{contEc}, since the typical length-scale of $\gamma V (\gamma \, \cdot)$ is $1/\gamma$. As $\gamma \to \infty$, we have $\gamma V (\gamma \, \cdot) \to (\int_\R V) \delta_0$, which connects to the limiting energy $E$ \eqref{limE3}. This implies that fluctuations on the length-scale $1/\gamma$ contribute to the energy difference \eqref{intro:firstorder}, which suggests that the boundary layer is likely to have the same length-scale.

The boundary layer having length-scale $1/\gamma$ agrees with the formal analysis of Hall's~\cite{Hall11} for the \emph{discrete} problem with $\gamma_n = \sqrt n$. Figure~\ref{fig:tilde:rho:n2} illustrates this scaling, by showing how spatially-rescaled densities constructed from the discrete minimisers are similar close to the lock.
\begin{figure}[h]
\labellist
\pinlabel \scriptsize $n=2^4$ at 326 250
\pinlabel \scriptsize $n=2^6$ at 325 229
\pinlabel \scriptsize $n=2^8$ at 325 208
\endlabellist
\begin{center}
\includegraphics[width=3in]{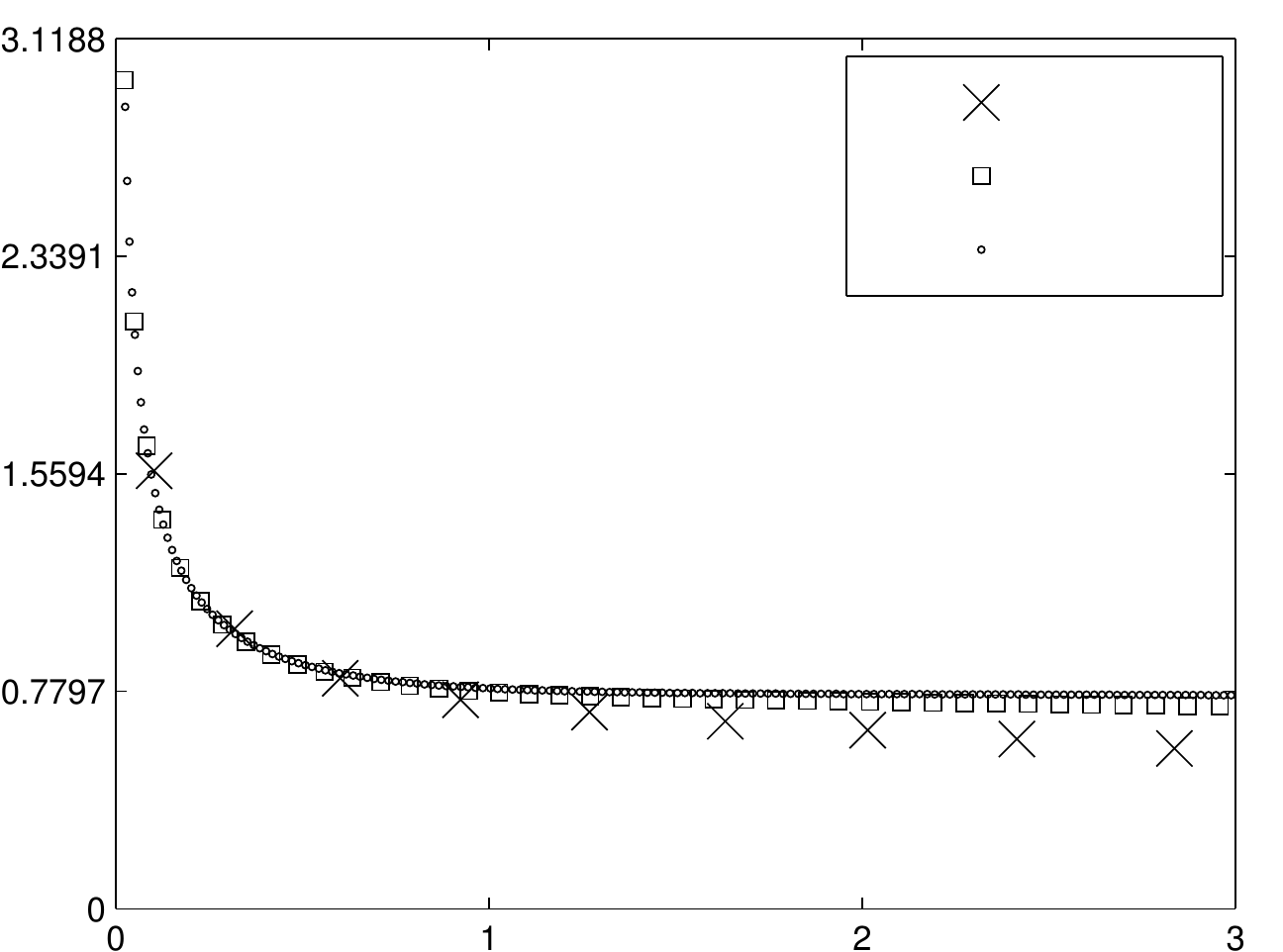}
\end{center}
\caption{Rescaled discrete densities $\gamma_n \rho_n$ (see~\eqref{intro:rhon}) plotted against the rescaled minimisers $\gamma_n x^n_\ast$, for $\gamma_n = \sqrt{n}$ and for different values of $n$. Close to the lock the data points seem to lie on a single curve, confirming Hall's prediction that the length-scale of the boundary layer is $O(1/\sqrt n) =O(1/\gamma_n)$. }
\label{fig:tilde:rho:n2}
\end{figure}

To capture effects on the length-scale $1/\gamma$ we zoom into $x=0$ by a factor $\gamma$. Since this rescaling operation will be used many times we introduce a corresponding notation: 
given a measure $\mu$, the rescaled measure $\rscld\mu$ is the measure 
\begin{equation}
\label{def:scalingOperator}
\rscld \mu := \gamma \; (x\mapsto \gamma x)_\#\mu,
\qquad \text{or, in terms of Lebesgue densities,}\qquad
(\rscld \mu)(x) := \mu(x/\gamma).
\end{equation}
This transformation zooms into the origin at rate $\gamma$, but preserves the amplitude of the (Lebesgue density of) the measure. The inverse transformation is written as $\antirscld \mu$.

Given a measure $\mu$, the appropriate zoomed-in version of $\mu$ describing the boundary layer is the measure obtained by blowing up the difference $\mu-\rhomin$,
\begin{equation}
\label{def:nu}
\nu := \rscld{[\mu-\rhomin]}.
\end{equation}
By definition, $\nu$ integrates up to zero over $[0, \infty)$, and it is bounded from below by $-\rscld \rhomin$. These properties define the admissible class $\mathcal{A}_\gamma$ in \eqref{def:Ac}. 
Note that $\nu$, unlike $\rscld \mu$, does not see the mismatch of the rescaled densities outside the domain $(0,1)$, shown in Figure \ref{fig:tilde:rho:n2}.

We rewrite (see Section \ref{sub:firstorder}) the energy \eqref{intro:firstorder} in terms of $\nu$ as 
\begin{multline} \label{def:Fc}
F_\gamma (\nu) := \frac12 \int_0^\infty \int_0^\infty V(x-y) \nu(dx)  \nu(dy) 
- \frac{1}{\sqrt{a}} \int_0^\infty \left( \int_x^\infty V(y) dy \right) \nu(dx) \\
+ \frac1{2 a \gamma} \int_0^\infty  \int_0^{\infty} \big[ V \left( y+ \left( 2\gamma\sqrt a - x \right) \right) - V(y+x) \big] y \, dy \, \nu(dx),
\end{multline}
where $a=\int_0^\infty V$, and all integrals (here and everywhere in the paper) denote integration over closed intervals.

The main result (see Theorem~\ref{thm:first:order:Gamma:conv}) is the characterisation of the limit behaviour of $F_\gamma$ as $\gamma\to\infty$. The main challenge in establishing this limit lies in controlling the interaction term, i.e.,~the first term in the right-hand side of \eqref{def:Fc}. Indeed, the admissible measures $\nu \in \mathcal{A}_\gamma$ are obtained by blowing up probability measures, and consequently we cannot expect their total variation to stay bounded in the limit $\gamma\to\infty$. Therefore, it is not clear how the interaction term can be defined for such $\nu$. For this reason, we impose non-negativity of the Fourier transform of $V$, which guarantees the interaction term to be non-negative. Moreover, this assumption allows for rewriting (see Lemma \ref{lem:T}) the interaction term as
\begin{equation*}
  \frac12 \int_0^\infty \int_0^\infty V(x-y) \nu(dx) \nu(dy)
  = \frac12\int_{\R} (T \nu)^2 (x) \, dx,
\end{equation*}
where $T$ is the linear operator (defined in \eqref{operator:T:on:Xp}) that satisfies $T^2 f = V \ast f$ for any $f$ in a suitable space. We think of $T$ as the `convolutional square root' of $V$. This operator provides us with the natural framework for our main theorem:

\begin{thm}\label{intro:thm}
The functionals $F_\gamma$ $\Gamma$-converge as $\gamma\to \infty$, with respect to the vague topology, to the energy $F$ defined as 
\begin{equation}
\label{def:F}
  F : \mathcal{A} \to \R,
  \qquad F(\nu) := 
  \frac12 \int_\R (T \nu)^2 (x) \, dx 
  - \frac{1}{\sqrt{a}} \int_0^\infty \left( \int_x^\infty V(y) dy \right) \nu(dx),
\end{equation}
where $\mathcal{A}$ is the admissible class of measures defined in \eqref{for:dom:F}. In addition, a sequence~$\nu_\gamma$ with $\sup_\gamma F_\gamma(\nu_\gamma)< \infty$ is compact in the vague topology.
\end{thm}

\noindent
We recall that convergence in the vague topology is defined as convergence against continuous functions with compact support. This topology is a natural choice, because we cannot expect the total variation of $\nu_\gamma$ to stay bounded in the limit $\gamma\to\infty$, but only locally bounded.

As a consequence of Theorem \ref{intro:thm} the functional $F$ achieves its minimum in $\mathcal A$; since $\mathcal A$ is convex and $F$ is strictly convex, this minimum is also unique.

\subsection{Approximation of the minimiser of $E_\gamma$ by `matching'.} The $\Gamma$-convergence result in Theorem~\ref{intro:thm} suggests an improved approximation of the minimiser of the energy $E_\gamma$ at the left boundary of the pile-up domain, where the \textit{bulk} density $\rhomin$ fails to describe the profile of the discrete density (see Figure \ref{DC3}).

Denoting with $\nu_*$ the minimiser of $F$, $\rhomin(0)+\nu_*$ is the \textit{blown-up} boundary layer profile, which corresponds to the behaviour of the minimiser of $E_\gamma$ close to the lock. In view of the $\Gamma$-convergence result, we can therefore define a `matched' continuous density in terms of the original, unscaled variables, as
\begin{equation} \label{for:defn:rhominbeta}
\rhomin^{\gamma}:= \rhomin + \antirscld\nu_\ast, \qquad\text{or, in terms of Lebesgue densities,}\qquad
\rhomin^\gamma (x) = \rhomin(x) + \nu_*(\gamma x),
\end{equation}
as the improved approximation of the minimiser of $E_\gamma$ in \eqref{contEc}. This expression appears also to be a good approximation of the \emph{discrete} optimal density $\rho_n$, as shown in Figure \ref{case3}. The agreement between $\rhomin^\gamma$ and the discrete density is striking, even for a small number of dislocation walls, except for the free end of the pile-up region, where a second boundary layer appears, whose analysis is beyond the scope of this paper.

\setlength{\unitlength}{0.75cm}
\begin{figure}[h]
\centering
\small
\subfigure[]{%
\includegraphics[width=2.31in]{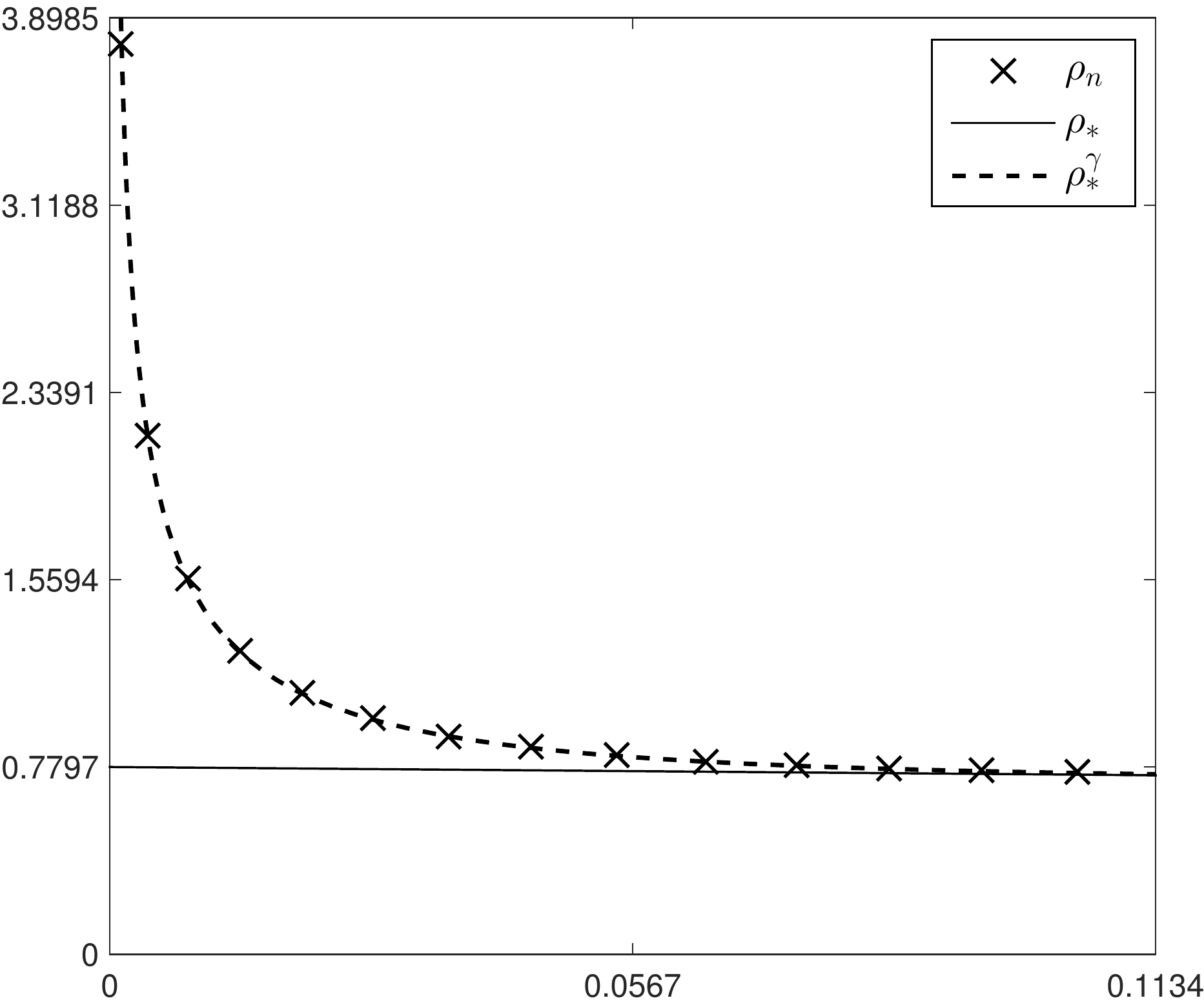}}
\quad \quad \hspace{.01cm}
\subfigure[]{%
\includegraphics[width=2.7in]{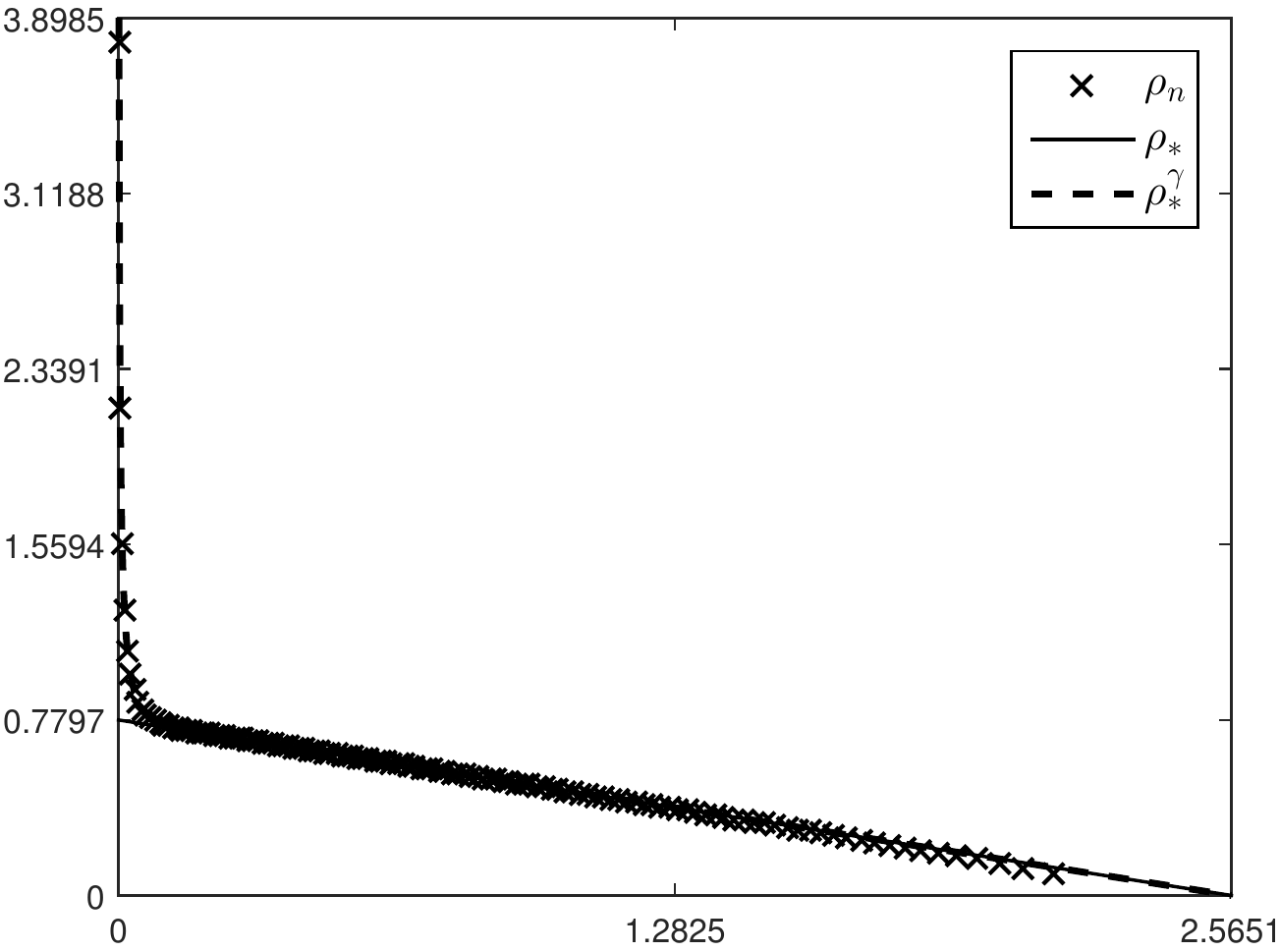}}
\put(-9.20,.9){\line(0,1){6.09}}
\put(-7,.9){\line(0,1){6.09}}
\put(-9.2,6.98){\line(1,0){2.2}}
\put(-9.2,.91){\line(1,0){2.2}}
\put(-9.21,.91){\line(-5,-2){1.41}}
\put(-9.22,6.98){\line(-5,-2){1.39}}
\caption{Comparison between the discrete pile-up profile $\rho_n$ in \eqref{intro:rhon}, the `bulk' density $\rhomin$ and the `matched' density  $\rhomin^{\gamma}=\rhomin + \gamma_\leftarrow\nu_\ast $, for $n = 2^7$ and $\gamma=\sqrt{n}$. The length of the $x$-axis in the left plot equals $\sqrt a / \gamma$.}
\label{case3}
\end{figure}
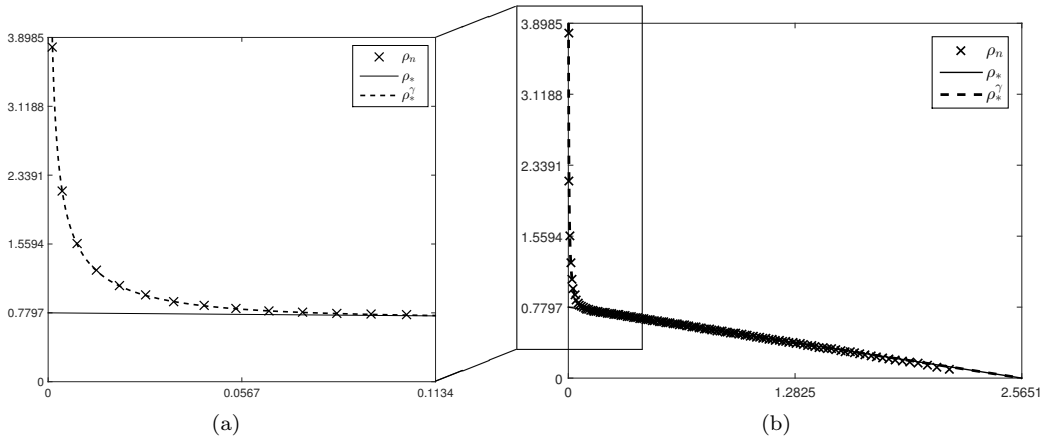

\subsection{Conclusion and comments.}

Theorem~\ref{intro:thm} gives a clear description of the boundary-layer behaviour of the pile-up at the lock through the minimiser of the limit energy. This theorem however only describes the boundary layer for the \emph{continuous} energy $E_\gamma$, while the `real' problem is discrete and described by $E_n$. Still, the predictions we obtain from Theorem~\ref{intro:thm} are remarkably accurate (see Figure~\ref{case3}). We now comment on some of the aspects of the result and the proof. 

\emph{Conditions on $V$.} The $\Gamma$-convergence of $E_n$ to $E$ is proved in \cite{GeersPeerlingsPeletierScardia13} under the assumptions that $V$ is even, of class $C^1$, integrable on $\R$, and decreasing and convex on the positive real line. As a consequence of these properties, $V$ is non-negative. A careful study of the proof in \cite{GeersPeerlingsPeletierScardia13} reveals that the convexity of $V$ can be relaxed to non-negativity of the Fourier transform of $V$ and the property that $V$ can be approximated from below in $L^1(\R)$ by continuous, positive definite functions\footnote{To see that convexity and regularity of $V$ implies $\widehat V \geq 0$, note that $\widehat V (\omega) = 2 \int_0^\infty V(x) \cos  2\pi\omega x \, dx = -(\pi\omega)^{-1} \int_0^\infty V'(x) \sin  2\pi\omega x \, dx$; this expression can be seen to be positive by considering each period of the sine, and using the monotonicity of $V'$.}. These properties appear again as natural requirements for the proof of Theorem \ref{intro:thm}, even though the setting contains no discreteness. 

In addition, we require that $V$ has finite first moment and that $\widehat V > 0$. The bound on the first moment is necessary to control the contribution to the energy difference in \eqref{intro:firstorder} which is related to the bulk of the full pile-up. Positivity of the Fourier transform (in addition to $\widehat V \geq 0$) is a technical requirement for gaining enough control on $\int (T \nu)^2$ in \eqref{def:F}.

\emph{Boundary layers in the discrete system.} The ultimate goal of this work is to state and prove an analogue of Theorem \ref{intro:thm} for the discrete energy $E_n$ in the limit $n \to \infty$. The current results can be seen as a first step towards such a convergence result, in which we tackle the scaling and compactness issues, but side-step the discreteness. 

\emph{Unexpected `dip' in the minimiser $\nu_\ast$.} Figure~\ref{fig:long:range:disc:vs:ct} gives another view on the comparison between the discrete and continuous densities, showing the boundary layers at both ends of the pile-up region in the same figure. Note that $\nu_\ast$ attains \emph{negative} values in the region where the boundary layer matches with the bulk density $\rho_\ast$. This `dip' indicates that in this region the dislocation walls tend to separate slightly farther than predicted by the bulk density. It is unclear to us why such a `dip' occurs. 

\setlength{\unitlength}{0.75cm}
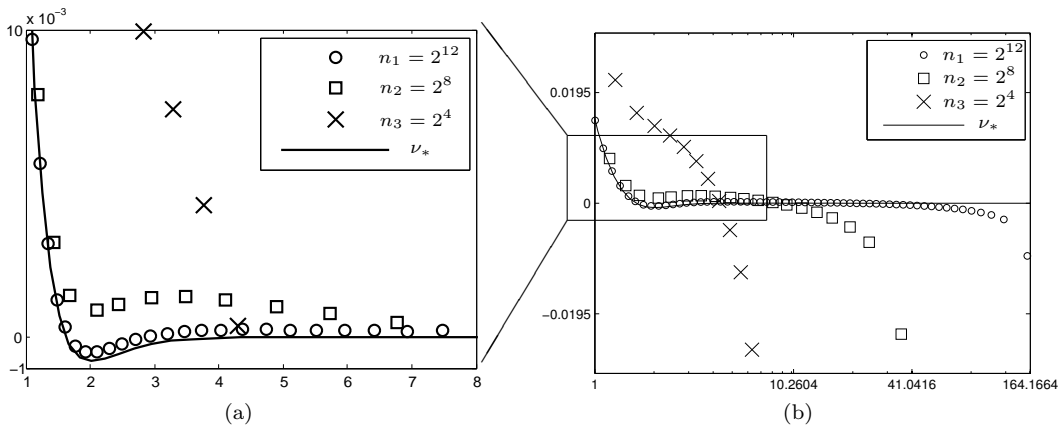
\begin{figure}[h]
\centering
\small
\subfigure[]{%
\labellist
\pinlabel \scriptsize$n_1=2^{12}$ at 414 521
\pinlabel \scriptsize $n_2=2^8$ at 412 501
\pinlabel \scriptsize $n_3=2^4$ at 412 479
\pinlabel \scriptsize $\nu_\ast$ at 415 458
\endlabellist
\includegraphics[width=2.5in]{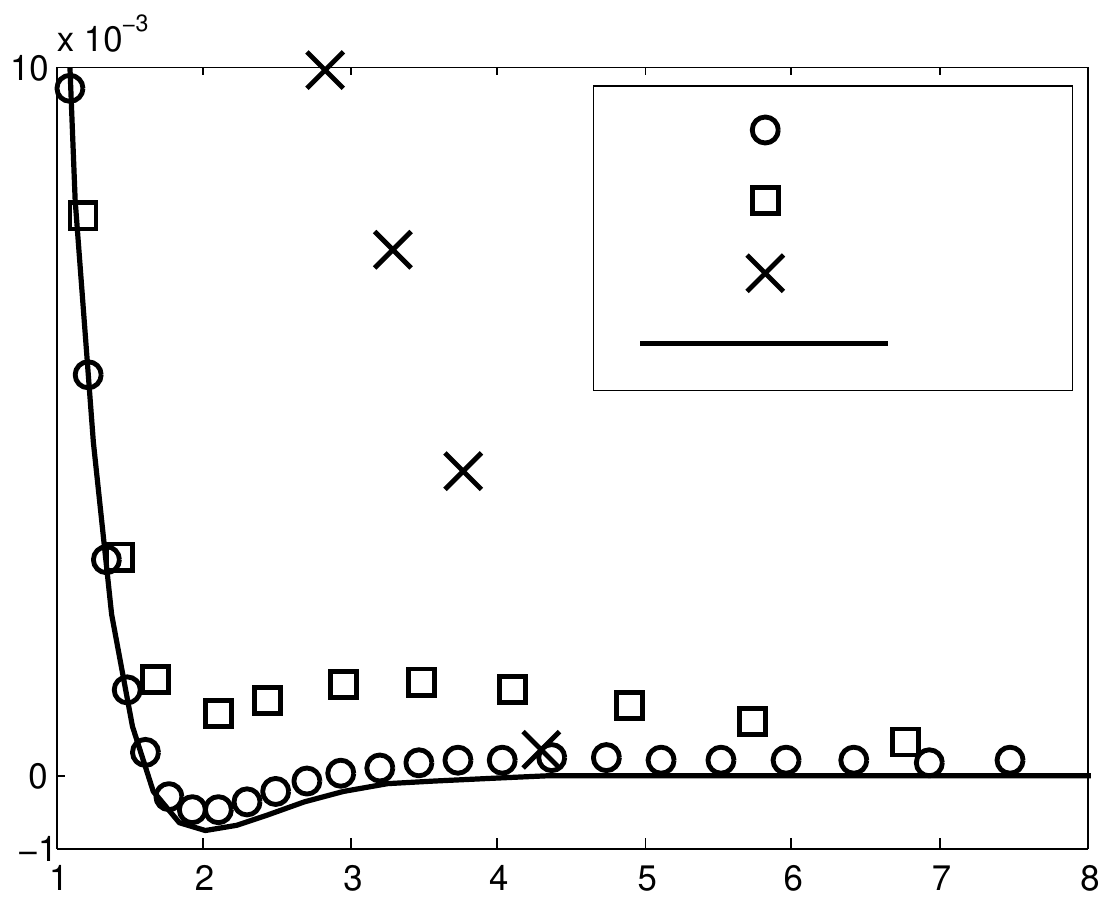}}
\qquad
\subfigure[]{%
\labellist
\pinlabel \scriptsize$n_1=2^{12}$ at 331 254
\pinlabel \scriptsize $n_2=2^8$ at 328 237
\pinlabel \scriptsize $n_3=2^4$ at 328 219
\pinlabel \scriptsize $\nu_\ast$ at 335 200
\endlabellist
\includegraphics[width=2.7in]{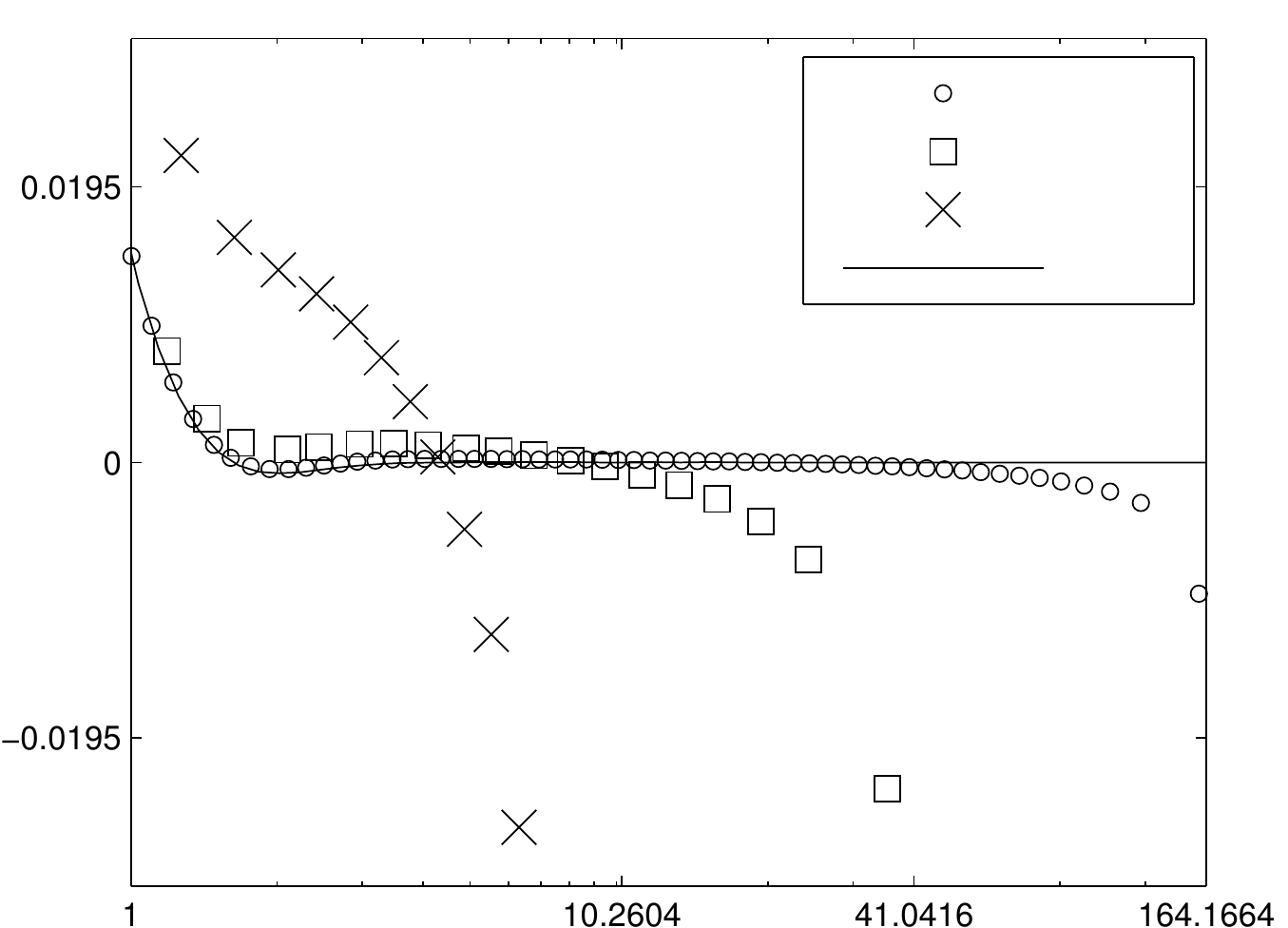}}
\put(-8.7,4.5){\line(0,-1){1.5}}
\put(-5.2,4.5){\line(0,-1){1.5}}
\put(-8.7,4.5){\line(1,0){3.5}}
\put(-8.7,3){\line(1,0){3.5}}
\put(-8.7,4.5){\line(-3,4){1.5}}
\put(-8.7,3){\line(-3,-5){1.5}}
\caption{Comparison between the discrete density $\nu_n$ and $\nu_\ast$ for $\gamma_n = \sqrt n$ and for different values of $n$. Many points are omitted to enhance readability. See Section \ref{Min} for a detailed description of these plots.}
\label{fig:long:range:disc:vs:ct}
\end{figure}

\medskip

In conclusion, Theorem \ref{intro:thm} together with the numerical illustrations predict accurately the local aggregation of dislocation walls at a lock. This aggregation is not recovered by the existing continuum models for dislocation pile-ups \cite{EversBrekelmansGeers04, Hall11, GeersPeerlingsPeletierScardia13} in the corresponding parameter regime $1 \ll \gamma_n \ll n$, because these models only predict the affine density profile of the bulk behaviour. Hence, these continuum models break down at a mesoscopic length-scale (given by $\mathcal O (1/\gamma_n)$ for $\gamma_n \ll \sqrt n$, based on numerical computations) near the lock, which is larger than the microscopic length-scale of neighboring dislocation walls, but smaller than the macroscopic length-scale of the full pile-up. As a result, these continuum models underestimate the stress around the lock. 

Theorem \ref{intro:thm} provides a way to fix this error by taking the effect on the mesoscale into account. Our results imply that the length-scale of the dislocation density is too coarse to describe the behaviour of dislocations close to grain boundaries. However, by taking into accounts effects on an intermediate mesoscopic length-scale, we can obtain an accurate description of the dislocations distribution in the pile-up region. We believe that such a mesoscopic description can be used to improve continuum models for the dislocation density at grain boundaries.

\medskip
The plan of the paper is as follows. In Section~\ref{sec:zero-order} we prove the convergence of $E_\gamma$ to $E$ (i.e., the zero-order approximation of $E_\gamma$); we then prove the main result, Theorem~\ref{intro:thm}, in Section~\ref{sect:firstorder}. In Section~\ref{sec:numerics} we discuss the agreement between discrete minimisers and the matched continuum density $\rhomin^{\gamma}$ \eqref{for:defn:rhominbeta} by means of numerical computations. Finally, Appendix \ref{app:T} and \ref{app:Step:1} are devoted to some technical steps in the proofs of results in previous sections.

\subsection{Assumptions and notation}

Although the expression~\eqref{defV} is the inspiration for this paper, and although we use~\eqref{defV} in the numerical calculations, the convergence results hold for a broader class of functions $V$. In Sections~\ref{sec:zero-order} and \ref{sect:firstorder} we make the following assumptions on $V$:
\begin{itemize}
\item[(V1)] $V:\R\to\R$ is even, and decreasing in $(0,\infty)$.
\item[(V2)] $V \in L^1(\R)$ has finite first moment, i.e., $\int_{\R} |V(x)|\,dx, \, \int_\R |x|V(x)\, dx< \infty$. 
\item[(V3)] $V$ has positive Fourier transform, which satisfies
\begin{equation*}
  \sqrt{ \widehat{V} } \in W^{2,\infty} (\R).
\end{equation*}
\item[(V4)] There exists a sequence $(V_k) \subset C_b (\R)$ of functions with non-negative Fourier transform such that $V_k \uparrow V$ pointwise a.e.~and $\|V - V_k\|_{L^1 (\R)} \to 0$ as $k \to \infty$.
\end{itemize}
Note that by (V1)--(V3), $V$ is non-negative and $\| \widehat{V} \|_\infty = \widehat{V} (0) = \int_{\R} V(x) \, dx < \infty$. We need the technical condition (V3) to define the `convolutional' square-root of $V$ by the operator $T$ as introduced in Theorem \ref{intro:thm} (see Section \ref{sub:firstorder}), and we rely on (V4) to rewrite the interaction term in \eqref{def:Fc} in terms of $T$.

\bigskip

Here we list some symbols and abbreviations that we use throughout the paper.

\begin{minipage}{15cm}
\begin{small}
\bigskip
\begin{tabular}{lll}
$a$ & $a = \int_0^\infty V=\frac12\int_{\R}V$ (${}= \pi^2/6$ in the case of~\eqref{defV})\\
$\mathcal A$, $\mathcal A_\gamma$ & admissible sets for $F$  and $F_\gamma$& \eqref{for:dom:F}, \eqref{def:Ac}\\
$\gamma_\to\mu$, $\gamma_\leftarrow\mu$ & transforms of $\mu$ by scaling space by $\gamma>0$ & \eqref{def:scalingOperator}\\
$E_n$ & discrete energy & \eqref{discreteEcn}\\
$E_\gamma$ & continuous energy at finite $\gamma$ & \eqref{contEc}\\
$E$ & (zero-order) limiting energy & \eqref{limE3}\\
$\widehat f$, $\mathcal{F}(f)$ & Fourier transform of $f$; $\mathcal{F}(f)(\omega) = \widehat f (\omega) := \int_\R f(x) e^{-2\pi ix\omega} \, dx$ \\
$\mathcal{F}^{-1}(f)$ & inverse Fourier transform of $f$;\\
$F_\gamma$ & first-order energy at finite $\gamma$ & \eqref{def:Fc}\\
$F$ & limiting first-order energy & \eqref{def:F}\\
$h_\gamma$ & auxiliary function in the expression of $F_\gamma$ & \eqref{hc:conv}\\
$\mathcal{H}^1$ & one-dimensional Hausdorff measure\\
$\mathcal{M}([0,\infty))$ & signed Borel measures on $\R$ with support in $[0,\infty)$\\
$\nu^+, \nu^-$ & positive and negative part of a measure $\nu \in \mathcal{M}([0,\infty))$; $
\nu^\pm\geq0$\\
$\mathcal{P}([0,\infty))$ &  non-negative Borel measures on $[0,\infty)$ of mass $1$ \\
$\rhomin$ & minimiser of $E$ & \eqref{def:rhomin}\\
$\tilderhomin$ & rescaled version of $\rho_*$; $\tilderhomin := \gamma_\to \rho_*$ &Sec.~\ref{sec:scalings} and \eqref{def:rhominstar}\\
$T$ & `convolutional square root' of $V$; $T^2 f = V \ast f$ & \eqref{operator:T:L2}, \eqref{operator:T:on:Xp} \\
$X$ & Hilbert space on which $T : X \to X$ is bounded, and $\mathcal A \subset X'$ & \eqref{for:defn:X}\\
$\|\cdot\|_q$ &$L^q$-norm, for $1\leq q\leq \infty$\\
$\|\cdot\|_{TV}$ &total variation norm of a signed Borel measure.
\end{tabular}
\end{small}
\end{minipage}
\medskip
For $\mu \in \mathcal{M}([0,\infty))$, we define $\int_a^b f \, d\mu := \int_{[a,b]} f \, d\mu$.

%

\section{Zero-order Gamma-convergence}
\label{sec:zero-order}
In this section we derive the zero-order term of the $\Gamma$-expansion of $E_\gamma$ in \eqref{contEc} in powers of $1/\gamma$. This term is the $\Gamma$-limit of $E_\gamma$ as $\gamma\to\infty$, in the \emph{narrow} or \emph{weak} topology (i.e.,~convergence against continuous and bounded functions). Not surprisingly, it turns out to be the same as the $\Gamma$-limit of the discrete energy $E_n$ in \eqref{discreteEcn}, namely the continuum energy $E$ defined in \eqref{limE3}. This is proved in the following theorem.
\begin{thm} \label{thm:Gamma:conv:E2c}
The $\Gamma$-limit of $E_\gamma$ as $\gamma\to\infty$ with respect to the narrow topology is given by 
\begin{equation*}
E(\mu) = \begin{cases} \displaystyle{a \int_0^\infty \rho(x)^2\, dx+ \int_0^\infty x\rho(x)\, dx,} \quad & \textrm{if}\,\, \mu(dx) = \rho(x)\, dx,\\
+\infty &  \textrm{otherwise}\,,
\end{cases}
\end{equation*}
where $a := \int_0^\infty V$. In addition, sequences with bounded $E_\gamma$ are compact in the narrow topology.
\end{thm}

\begin{proof}
Compactness, i.e., tightness, is a  consequence of the inequality $E_\gamma(\mu)\geq \int_0^\infty x\, \mu(dx)$. The proof of the liminf inequality follows from a simple modification of the proof of the discrete-to-continuum convergence in \cite[Theorem~7]{GeersPeerlingsPeletierScardia13}. 

As for the limsup inequality, we first reduce by density arguments to $\rho \in C_c( [0,\infty ))$. Then we take as recovery sequence the constant sequence $\rho_\gamma := \rho$. Note that $V_\gamma (x, y) := \gamma V (\gamma (x-y))$ converges vaguely to $2 a \mathcal H^1 \llcorner D$, where $D$ is the diagonal of the first quadrant, 
$$
D := \{ (x,y) \in (0, \infty)^2 \, | \, x = y \}.
$$
Then, we conclude the limsup inequality from
\begin{align*}
    \lim_{\gamma \rightarrow \infty} \int_0^\infty \int_0^\infty V_\gamma (x,y) \rho(x) \rho (y) \, dx \, dy
    = 2a \int_0^\infty \int_0^\infty \rho(x) \rho (y) \mathcal H^1 \llcorner D ( dx\,dy )
    = 2a \int_0^\infty \rho^2.
\end{align*}
\end{proof}

\begin{rem} \label{rem:minz:E3}
We recall that the minimiser of $E$ is given by 
\begin{equation}
\label{def:rhomin}
 \rhomin (x) := \left(\frac 1{\sqrt a} - \frac x{2 a}\right)^+, \qquad E ( \rhomin ) = \frac 43 \sqrt a\,,
\end{equation}
hence the rescaled version of $\rhomin$ is, according to the definition \eqref{def:scalingOperator},
\begin{equation}
\label{def:rhominstar}
 \tilderhomin (x) := \rscld{\rhomin}(x) = \left(\frac 1{\sqrt a} - \frac x{2  \gamma a}\right)^+.
\end{equation}
\end{rem}


\section{First-order Gamma-convergence}\label{sect:firstorder}
In this section we prove Theorem~\ref{intro:thm} (see Theorem~\ref{thm:first:order:Gamma:conv}), which states the $\Gamma$-convergence of $F_\gamma$ in~\eqref{def:Fc} to $F$ in~\eqref{def:F}. 

\subsection{Preliminary results}\label{sub:firstorder}

\emph{Rewriting the energy.} Here we show that under the blow-up transformation~\eqref{def:nu}, which transforms $\mu$ into $\nu$, the energy $E_\gamma$ converts into the expression for $F_\gamma$ in~\eqref{def:Fc}. 

Setting $\tilde \mu^\gamma:= \rscld \mu$, we write
$$
\gamma E_\gamma(\mu) = \frac12 \int_0^\infty \int_0^\infty V(x-y)\tilde \mu^\gamma(dx)\tilde\mu^\gamma(dy) + \frac1\gamma \int_0^\infty x\tilde\mu^\gamma(dx)=: \tilde E_\gamma(\tilde \mu^\gamma).
$$
Note that $\int_0^\infty\tilde\mu^\gamma =\gamma$. By defining as in \eqref{def:nu} $\nu := \tilde\mu^\gamma - \tilde\rho_\ast$ we write, by using Fubini's theorem,
\begin{align}
F_\gamma (\nu) &:= \gamma\left(E_\gamma(\mu) - E_\gamma(\rhomin)\right) = \tilde E_\gamma(\tilde \mu^\gamma) - \tilde E_\gamma(\tilderhomin) \nonumber\\
&= 
\frac12 \int_0^\infty \int_0^\infty V(x-y) \nu(dx)  \nu(dy) + \int_0^\infty \left(\int_0^\infty V(x-y)\rhomin(y/\gamma)dy + \frac x \gamma \right) \nu(dx).
\label{eqn:Fc:untidy} 
\end{align}
For the sake of notation, we omit the superscript $\gamma$. The measures $\nu$ obtained via the rescaling~\eqref{def:nu} belong to the admissible class $\mathcal{A}_\gamma$ defined as
\begin{equation}
\label{def:Ac}
\mathcal{A}_\gamma
:= 
\left\{\nu \in \mathcal M([0,\infty))
: 
\int_0^\infty |\nu| \leq 2\gamma, 
\quad \int_0^\infty \nu=0,
\quad \nu (dx) \geq -\tilderhomin (x) dx \right\}.
\end{equation}
Whenever convenient, we extend $\nu$ to $\R$ by zero without changing notation. Note that by~\eqref{def:Ac} and~\eqref{def:rhominstar}, 
\begin{equation}
\label{prop:supp-nu-}
\supp\nu^-\subset \supp\tilderhomin = [0,2\gamma\sqrt a] \quad\text{for each $\nu\in \mathcal A_\gamma$}.
\end{equation}

Next we rewrite $F_\gamma (\nu)$. By using the explicit expression of~$\rhomin$ in \eqref{def:rhomin} and the fact that $V$ is even, the integrand of the second term in \eqref{eqn:Fc:untidy} can be cast into
\begin{align}\label{rewriting_int:2}
  &\int_0^\infty V(x-y)\rhomin(y/\gamma)dy + \frac x\gamma \nonumber\\
  &= \frac1{\sqrt a} \int_0^{2\gamma\sqrt{a}} V(x-y)dy 
     -\frac1{2a\gamma}\int_0^{2\gamma\sqrt{a}} V(x-y)ydy 
     + \frac1{2a\gamma}\int_{\R} V(x-y)xdy \nonumber\\
  &= \frac1{\sqrt a} \int_0^{2\gamma\sqrt{a}} V(x-y)dy  
     + \frac1{2a\gamma}\int_{\R} (x-y)V(x-y) dy 
     + \frac1{2a\gamma}\int_{ \R{} \setminus (0, 2\gamma\sqrt{a}) } V(x-y)ydy \nonumber\\
  &= \frac1{\sqrt a} \int_0^\infty V(x-y)dy  
      + \frac1{2a\gamma}\int_{-\infty}^0 V(x-y)ydy
     + \frac1{2a\gamma}\int_{2\gamma\sqrt{a}}^\infty V(x-y) (y - 2\gamma\sqrt{a}) dy \nonumber\\
  &= 2\sqrt a - \frac1{\sqrt a} \int_x^\infty V(y)dy  
     + h_\gamma(x),
\end{align}
where we have set
\begin{align}\label{hc:conv}
h_\gamma (x)
:= \frac1{2 a \gamma} V \ast \left( ( x \wedge 0 ) \vee (x - 2 \gamma \sqrt a ) \right) 
= \frac{1}{2a\gamma} \int_0^{\infty} \big[ V \left( y+ \left( 2\gamma\sqrt a - x \right) \right) - V(y+x) \big] y \, dy.
\end{align}
Substituting \eqref{rewriting_int:2} in \eqref{eqn:Fc:untidy} and using the fact that $\int_0^\infty \nu = 0$, we obtain
\begin{gather} \label{Fc}
F_\gamma (\nu) 
= \frac12 \int_0^\infty  (V \ast \nu)(x) \nu(dx) 
- \int_0^\infty g(x) \nu(dx)
+ \int_0^\infty h_\gamma(x) \nu(dx), \\\label{eqn:def:g}
\text{where } \: g(x) := \frac1{\sqrt a} \int_x^\infty V(y)dy.
\end{gather} 
Note that, due to assumptions (V1) and (V2), $g \in C_b ([0,\infty)) \cap L^1 (0, \infty)$ is non-negative and non-increasing.

\bigskip

\textit{Properties of the auxiliary function $h_\gamma$.} In the following lemma we derive some useful estimates for
$h_\gamma$. The qualitative behaviour of $h_\gamma$ is illustrated in Figure \ref{fig:plot:hc}.

\begin{lem} \label{lem:hc}
The function $h_\gamma$ defined in \eqref{hc:conv} is non-decreasing and satisfies the bounds
\begin{align}
&|h_\gamma| \leq C\frac{1}{\gamma} \quad \,\,\, \textrm{in}  \quad \left[0, 2\gamma\sqrt{a}\right],\label{hc1}\\
&|h_\gamma| \leq \frac{\sigma(\gamma)}\gamma \quad \textrm{in} \quad \left[\sqrt \gamma, 2\gamma\sqrt{a}-\sqrt \gamma\right],\label{hc2}
\end{align}
for some $\sigma(\gamma)\to 0$ as $\gamma\to \infty$.
\end{lem}

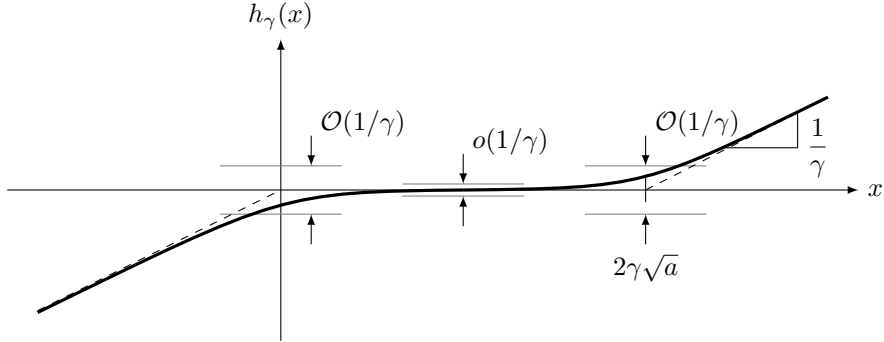
\begin{figure}[h!]
\centering
\begin{tikzpicture}[scale=0.8, >= latex]
\draw[->] (-4.5,0) -- (9.5,0) node[right] {$x$};
\draw[->] (0,-2.5) -- (0,2.5) node[above] {$h_\gamma(x)$};

\draw (6,0.2) -- (6,-0.2);

\draw[color=gray] (-1, 0.4) -- (1, 0.4);
\draw[color=gray] (-1, -0.4) -- (1, -0.4);
\draw[->] (0.5,0.9) -- (0.5,0.4);
\draw (0.5,1.1) node[right] {$\mathcal O (1/\gamma)$};
\draw[->] (0.5,-0.9) -- (0.5,-0.4);

\draw[color=gray] (2, 0.1) -- (4, 0.1);
\draw[color=gray] (2, -0.1) -- (4, -0.1);
\draw[->] (3,0.6) -- (3,0.1);
\draw (3,0.8) node[right] {$o (1/\gamma)$};
\draw[->] (3,-0.6) -- (3,-0.1);

\draw[color=gray] (5, 0.4) -- (7, 0.4);
\draw[color=gray] (5, -0.4) -- (7, -0.4);
\draw[->] (6,0.9) -- (6,0.4);
\draw (6,1.1) node[right] {$\mathcal O (1/\gamma)$};
\draw[->] (6,-0.9)  node[below] {$2 \gamma \sqrt{a}$} -- (6,-0.4);

\draw (7.25,0.7) -- (8.5,0.7) node[right] {$\dfrac1\gamma$} -- (8.5, 1.3);

\draw[dashed] (-4,-2) -- (0,0);
\draw[dashed] (6,0) -- (9,1.54);
\draw[very thick] (-4,-2.03) .. controls (0,-0.06) and (0,-0.03) .. (3,0);
\draw[very thick] (3,0) .. controls (6,0.03) and (6,0.06) .. (9,1.54); 
\end{tikzpicture}
\caption{Qualitative plot of $h_\gamma$.}\label{fig:plot:hc}
\end{figure}

\begin{proof}
The function $h_\gamma$ is non-decreasing, since it is the convolution of a non-negative function with a non-decreasing function. 

To prove \eqref{hc1} we note that by the monotonicity of $h_\gamma$, for all $x\in [0,2\gamma\sqrt a]$,
\begin{equation}\label{hcbounds}
-C \frac1\gamma \leq h_\gamma (0) \leq h_\gamma(x)\leq  h_\gamma (2 \gamma \sqrt a) \leq C \frac1\gamma,
\qquad\text{with } C = \frac1{2a} \int_0^\infty V(y) y\, dy.
\end{equation}
As for \eqref{hc2}, we estimate
\begin{align*}
h_\gamma (2 \gamma \sqrt a - \sqrt \gamma) 
= - h_\gamma(\sqrt \gamma) 
\leq \frac{1}{2a\gamma} \int_0^\infty V(y+\sqrt \gamma) y \, dy 
\leq \frac{1}{2a\gamma} \int_{\sqrt{\gamma}}^\infty V(z) z \, dz,
\end{align*}
which implies \eqref{hc2} by assumption (V2).
\medskip
\end{proof}

\bigskip

\textit{The `convolutional square root' of $V$.} In~\cite[Appendix]{GeersPeerlingsPeletierScardia13} it was shown that, for the special case of~\eqref{defV},
\[
\widehat V(\omega) 
= \int_\R V(x) e^{-2\pi ix\omega} \, dx 
= \frac1{2\omega \sinh(\pi^2\omega)}\left(\cosh(\pi^2\omega) - \frac{\pi^2\omega}{\sinh(\pi^2\omega)} \right).
\]
We observe that $\widehat V$ is even, strictly positive on $\R$, and $p := \sqrt{\widehat{V}} \in W^{2,\infty} (\R)$ is strictly positive. However, since $\widehat V$ decays to zero at infinity at rate $1/|\omega|$, the function $p(\omega)$ decays at rate $|\omega|^{-1/2}$ and fails to be in $L^2(\R)$. Therefore, the inverse Fourier transform of $p$ may not be a function $W$, and it is not clear whether there exists a function $W:\R\to\R$ such that $V = W*W$. 

In this paper we also want to consider any $V$ satisfying conditions (V1)--(V4). We therefore circumvent these difficulties by defining `convolution with $W$' as an operator. We leave the details of the construction of this operator to Appendix \ref{app:T}, and focus here on its key properties. We consider the operator 
\begin{equation}\label{operator:T:L2}
T: L^2 (\R) \to L^2 (\R), \qquad Tf := \mathcal F^{-1} \big( p\widehat f \big).
\end{equation}
Since $p^2 = \widehat V$, we obtain 
\begin{equation*} 
  T^2 f = \mathcal F^{-1} \big( p^2 \widehat f \big) = \mathcal F^{-1} \big( \widehat V \widehat f \big) = V \ast f,   
\end{equation*}
which motivates us to think of $T$ as the `convolutional square root' of $V$. 
However, in view of the expression of $F$ in Theorem \ref{intro:thm}, we need to extend the domain of $T$ beyond $L^2 (\R)$. It turns out that a natural space for this extension is the dual of the Hilbert space
\begin{equation} \label{for:defn:X}
  X := \Big\{ f \in H^1(\R) : x^2 f(x) \in L^2(\R) \Big\},
  \quad (f,\phi)_X := \int f' \phi' + \int f \phi (1 + x^4).
\end{equation}
The space $X$ is chosen to be as large as possible to require little regularity on $p$, but small enough such that $X'$ contains the domain of $F$. Lemma \ref{lem:T}.\eqref{lem:T:BLSO} and Lemma \ref{lem:T}.\eqref{lem:T:on:A} below quantify this statement. 

Furthermore, Lemma \ref{lem:T} lists the properties of $T$ and $X$ which we need in our proof of Theorem \ref{intro:thm}. In particular, Lemma \ref{lem:T}.\eqref{lem:T:Tnu2:equals:Vnunu} applies to $\nu_\gamma \in \mathcal A_\gamma$ for all $\gamma > 0$ (by choosing $f=\tilderhomin$). Consequently, the first term of $F_\gamma (\nu_\gamma)$ in \eqref{Fc} equals the first term of $F(\nu_\gamma)$ in \eqref{def:F}. The proof of Lemma \ref{lem:T}.\eqref{lem:T:on:A} is a corollary of Lemma \ref{lem:g:nu:bounds}. We prove the remaining statements of Lemma \ref{lem:T} in Appendix \ref{app:T}.

\begin{lem} \label{lem:T}
The linear operator $T$ defined by \eqref{operator:T:L2} and the space $X$ defined by \eqref{for:defn:X} satisfy the following properties:
\begin{enumerate}[(i)]
  \item \label{lem:T:BLSO} $T$ is a symmetric bounded linear operator from $L^2 (\R)$ to itself, from $X$ to itself, and from $X'$ to itself, where we define
  \begin{equation}\label{operator:T:on:Xp}
    T : X' \to X' 
    \quad \text{by} \quad
\langle T\xi, f \rangle := \langle \xi, Tf \rangle \text{;}
\end{equation}
  \item \label{lem:T:T2:equals:Vast} For all $\xi \in X'$ it holds that $T^2 \xi = V \ast \xi$;
  \item \label{lem:T:on:A} $\displaystyle \Big\{\nu \in \mathcal M([0, \infty))
	   : \nu (dx) \geq -\rhomin (0) dx, 
	     \: \sup_{x \geq 0} \nu^+ ([x, x+1]) < \infty 
	     \Big\} \subset X'$;
  \item \label{lem:T:Tnu2:equals:Vnunu} For all $\nu \in \mathcal M ([0,\infty))$ for which there exists $f \in H^1(\R)$ with $x f(x) \in L^2(\R)$ such that $\nu + f \geq 0$ is a finite measure, it holds that 
  \begin{equation*}
    \int_\R (T \nu)^2(x) dx = \int_0^\infty (V \ast \nu)(x) \nu (dx).
  \end{equation*}
\end{enumerate}
\end{lem}

\bigskip

\emph{Domain $\mathcal A$ of the limit energy $F$.} We now give the complete definition of the limit functional~$F$: 
\begin{equation} \label{for:F}
  F : \mathcal A \to \R, \quad  F(\nu) = \frac12 \int_\R (T \nu)^2 (x) dx - \int_0^\infty g(x) \nu(dx),
  \end{equation}
where $g$ is given in~\eqref{eqn:def:g} and where the admissible class $\mathcal{A}$ is
 \begin{equation} \label{for:dom:F}
    \mathcal A 
    := \left\{ \nu \in \mathcal M([0, \infty)) : 
	     \nu (dx) \geq -\rhomin (0) dx, \quad
	     \sup_{x \geq 0} \nu^+ ([x, x+1]) < \infty, \quad
	     T \nu \in L^2(\R{})
	   \right\}.
  \end{equation}
As for $\mathcal A_\gamma$, we extend $\nu \in \mathcal A$ to $\R$ by zero whenever convenient. By `$T \nu \in L^2(\R{})$' we mean that the linear operator $T \nu : (X, \| \cdot \|_2) \to \R$ is bounded; Lemma \ref{lem:T}.\eqref{lem:T:BLSO},\eqref{lem:T:on:A} imply that this requirement is meaningful because of the lower and upper bounds on $\nu$ in $\mathcal A$.

While for finite $\gamma$ the set $\mathcal A_\gamma$ consists of measures with uniformly bounded total variation, in the limiting set $\mathcal A$ this bound has vanished, and \emph{a priori} it is unclear why a sequence $(\nu^\gamma) \subset \mathcal A_\gamma$ with bounded energy $F_\gamma (\nu^\gamma)$ could not acquire infinite total variation in the limit. Instead, our compactness result in Section \ref{subsection:main} implies $\nu \in \mathcal A$, which gives a \emph{local} bound on $|\nu|$, but no global bound. Maybe minimising sequences have bounded total variation, but we have not investigated this.

\subsection{Key estimates.} Lemma \ref{lem:cpness:est} states two estimates which are central to the proof of the compactness statement in Theorem~\ref{intro:thm}. Lemma \ref{lem:g:nu:bounds} shows how the local bounds of $\nu \in \mathcal A$ provide control over the linear term $\int g \nu$.

\begin{lem} \label{lem:cpness:est}
Let $\eta > 0$. For any bounded interval $[r, s] \subset [0, \infty)$ with $s - r \geq \eta$ there exist constants $C>0$ and  $C_\eta > 0$ (decreasing in $\eta$) such that for all $\gamma > 0$ and $\nu \in \mathcal A_\gamma$ the following estimates involving ${\lambda := T\nu}$ hold
\begin{align} \label{for:lem:cpness:est:nu}
  &\nu^+ ([r, s]) 
  \leq 
  C_\eta ( \sqrt{s-r} \| \lambda \|_{2}  + s-r ),
  \\ 
  \label{for:lem:cpness:est:lambda}
  &\int_{-\infty}^0 T\lambda = \int_{-\infty}^0 V\ast \nu
  \leq 
  C ( \| \lambda \|_{2} + 1 ).
\end{align}
\end{lem}

\begin{proof} It is not restrictive to prove the claim only for $\nu$ with $T\nu\in L^2(\R)$, the bounds \eqref{for:lem:cpness:est:nu} and \eqref{for:lem:cpness:est:lambda} being trivial otherwise.

Since $V$ is positive, we obtain
\begin{align*}
\int_r^s V \ast \nu^{+} 
= 
\int_r^s \int_0^\infty V(x-y) \nu^{+}(dy) dx 
\geq 
\int_r^s \left( \int_r^s V(x-y) dx \right) \nu^{+}(dy) 
\geq 
\tilde C_\eta\,\nu^+ ([r,s]),
\end{align*}
where $\tilde C_\eta$ is a strictly positive constant increasing in $\eta$ and converging to $0$ as $\eta \searrow 0$. Since we assume that $T\nu\in L^2(\R)$, it follows from Lemma \ref{lem:T}.\eqref{lem:T:BLSO} that $T^2\nu\in L^2(\R)$ and $\|T^2\nu\|_2 \leq C\|T\nu\|_2$. We then get the desired upper bound \eqref{for:lem:cpness:est:nu} from Lemma \ref{lem:T}.\eqref{lem:T:T2:equals:Vast} by
\begin{align*} 
\tilde C_{\eta} \nu^+ ([r,s])
&\leq 
\int_r^s V \ast \nu^{+} 
=
\int_r^s V \ast \nu + \int_r^s V \ast \nu^- 
\\
& = \int_r^s T^2\nu  + \int_r^s \int_0^\infty V(x-y)\nu^-(dy) \, dx \\
&\leq 
\sqrt{s-r}\, \| T^2\nu\|_2 + (s-r) \rhomin (0) \| V \|_1 \\
& \leq   C\sqrt{s-r}  \|T\nu \|_2 + (s-r) \rhomin (0) \| V \|_1. 
\end{align*}

For proving the second estimate \eqref{for:lem:cpness:est:lambda}, we use \eqref{for:lem:cpness:est:nu} with $s - r = \eta = 1$. Since $V$ is positive, even and increasing in $(-\infty,0)$, we have the estimate
\begin{align*} 
\int_{-\infty}^0 V \ast \nu
&\leq  
\int_{-\infty}^0 \int_0^\infty V(x-y) \nu^+ (dy) dx
\\
&\leq
\int_{-\infty}^0 \sum_{k=0}^\infty V(x - k) \nu^+ ([k, k+1]) dx
\\ 
&\leq
C ( \| T\nu \|_2 + 1 ) \sum_{k=0}^\infty \int_k^\infty V,
\end{align*}
where the sum in the right-hand side is finite due to assumption (V2).
\end{proof}

\begin{lem} \label{lem:g:nu:bounds}
There exists a constant $C > 0$ such that for all $\nu \in \mathcal M ([0,\infty))$ it holds that
\begin{align}
  \bigg| \int_0^\infty g \nu \, \bigg| 
  &\leq C \sup_{x \geq 0} | \nu | ([x, x+1]),
  \quad \text{where $g$ is as in \eqref{eqn:def:g}, and}  \label{for:lem:g:nu:bounds}\\\label{for:lem:g:nu:bounds:f}
  \bigg| \int_0^\infty f \nu \, \bigg| 
  &\leq C \|f\|_X \sup_{x \geq 0} | \nu | ([x, x+1])
  \quad \text{for all } f \in X.
\end{align}
\end{lem}

\begin{proof}
Let $\varphi \in \{g\} \cup X$. We set $I_k := (k, k+1)$ for $k \in \N$, and estimate
\begin{align}\label{l:intermediate}
\bigg| \int_{0}^{\infty} \varphi \nu \, \bigg|
\leq \sum_{k=0}^{\infty} \int_{k}^{k+1} |\varphi(x)| \, |\nu|(dx) 
\leq \Big( \sup_{x \geq 0} | \nu | ([x, x+1]) \Big) \sum_{k=0}^{\infty} \| \varphi \|_{L^\infty (I_k)}.
\end{align}
We estimate the right-hand side of \eqref{l:intermediate} for $\varphi = g$ and $\varphi \in X$ separately. Since $g \in L^1(0,\infty)$ is bounded and non-increasing, we obtain
\begin{equation*}
  \sum_{k = 0}^\infty \| g \|_{L^\infty (I_k)} 
  \leq g(0) + \int_0^\infty g
  < \infty,
\end{equation*}
hence \eqref{for:lem:g:nu:bounds}. Next, to prove \eqref{for:lem:g:nu:bounds:f}, we set $\varphi = f \in X$, and fix $k \in \N$. By continuity of $f$, there exists $x_k \in I_k$ at which $f^2(x_k) =\int_{I_k}f^2$. Therefore, by the Fundamental Theorem of Calculus, we obtain
\begin{equation*} 
  f^2 (y) = f^2(x_k) + \int_{x_k}^y (f^2(x))'dx
  = \int_{I_k}f^2 + 2 \int_{x_k}^y f f'
  \qquad \text{for all } y \in I_k.
\end{equation*}
Then, by using the Cauchy-Schwarz and Young inequalities, we estimate, for $k \geq 1$,
\begin{align} \notag
  \| f \|_{L^\infty (I_k)}^2
  &\leq \int_{I_k}f^2 + 2 \| f \|_{L^2 (I_k)} \| f' \|_{L^2 (I_k)} \leq  \frac{1}{k^4}\|x^2f\|^2_{L^2(I_k)} + \frac{2}{k^2} \|x^2 f \|_{L^2 (I_k)} \| f' \|_{L^2 (I_k)} \\\label{for:app:T:1}
  &\leq  \frac C{k^2} \left( \| x^2 f \|_{L^2 (I_k)}^2 + \| f' \|_{L^2 (I_k)}^2 \right)
  =: \frac{b_k}{k^2}.
\end{align} 
Note that $f \in X$ implies $(b_k)_k \subset \ell^1(\N{})$. We obtain the desired estimate, by using \eqref{for:app:T:1} and the Cauchy-Schwarz inequality on sequences, from
\begin{equation*} 
  \sum_{k=1}^\infty \|f \|_{L^\infty (I_k)}
  \leq \sum_{k=1}^\infty \frac1k \sqrt{b_k}
  \leq \left[ \sum_{k=1}^\infty \frac1{k^2} \right]^{\frac12} \left[ \sum_{k=1}^\infty b_k \right]^{\frac12} 
  \leq C \| f \|_{X},
\end{equation*} 
and from $\|f \|_{L^\infty (I_0)} \leq C \| f \|_X$, by the first inequality in \eqref{for:app:T:1}.
\end{proof}

\subsection{Main result: $\Gamma$--convergence}\label{subsection:main}

The following theorem is the main result in this paper. 

\begin{thm} (First-order $\Gamma$--convergence).\label{thm:first:order:Gamma:conv}
Any sequence $(\nu_\gamma) \subset \mathcal A_\gamma$ with bounded energies $F_\gamma (\nu_\gamma)$ is pre-compact in the vague topology. Furthermore, $F_\gamma$ $\Gamma$-converges to $F$ as $\gamma\to \infty$ with respect to the vague topology, i.e.,
\begin{subequations}
\label{for:thm:first:order:Gconv}
\begin{alignat}2
\label{for:thm:first:order:Gconv:liminf}
&\text{for all } \nu_\gamma \weakto \nu \text{ vaguely,} & \liminf_{\gamma\to\infty}  F_\gamma(\nu_\gamma) &\geq F(\nu), \\
&\text{for all }\nu \in \mathcal A\text{ there exists }\nu_\gamma\weakto \nu \text{ vaguely such that}\quad & \limsup_{\gamma\to\infty} F_\gamma(\nu_\gamma) &\leq F(\nu).
\label{for:thm:first:order:Gconv:limsup}
\end{alignat}
\end{subequations}
\end{thm}

\begin{proof} 
\textbf{Compactness.}  
Let $(\nu_\gamma)\subset \mathcal{A}_\gamma$ be such that $F_\gamma (\nu_\gamma) \leq C$. We split the proof into several steps.

\medskip

\textit{Step 1. $F(\nu_\gamma)$ is bounded.} We claim that $\int_0^\infty h_\gamma \nu_\gamma$ is bounded from below. This bound implies, together with Lemma \ref{lem:T}.\eqref{lem:T:Tnu2:equals:Vnunu} applied to $(\nu_\gamma)$, that
\begin{equation}\label{diff:en:nuc}
F (\nu_\gamma) 
= F_\gamma (\nu_\gamma) - \int_0^\infty  h_\gamma \nu_\gamma 
\leq C.
\end{equation}
From estimate \eqref{diff:en:nuc} we conclude in particular that finiteness of $F_\gamma (\nu_\gamma)$ implies finiteness of $F (\nu_\gamma)$. 

We now prove the claim. Since $h_\gamma$ is non-decreasing (Lemma \ref{lem:hc}) and by~\eqref{prop:supp-nu-} $\supp \nu_\gamma^-\subseteq~[0,2\gamma\sqrt{a}]$, we deduce the bound
\begin{equation*} 
  \int_0^\infty h_\gamma \nu_\gamma
  = 
  \int_0^\infty h_\gamma \nu_\gamma^+ - \int_0^{2\gamma\sqrt{a}} h_\gamma \nu_\gamma^-
  \geq 
 h_\gamma (0) \int_0^\infty \nu_\gamma^+ - h_\gamma (2 \gamma \sqrt a) \int_0^\infty \nu_\gamma^-
  \geq 
  -C,
\end{equation*}
where the last inequality follows from \eqref{def:Ac} and \eqref{hcbounds}.

\medskip

\textit{Step 2. Convergence of $\nu_\gamma$ and $T\nu_\gamma$.} From Step 1 we have 
\begin{equation}\label{en:g}
C\geq F(\nu_\gamma) =  \frac12 \| \lambda_\gamma \|_2^2 - \int_0^\infty g \nu_\gamma,
\end{equation}
where $\lambda_\gamma:=T\nu_\gamma$. We rewrite the second integral in the right-hand side of \eqref{en:g} in terms of $\lambda_\gamma$ by 
\begin{align*}
\int_0^\infty g\nu_\gamma 
&= \frac{1}{\sqrt a}\int_0^\infty \left(\int_x^\infty V(y) dy \right) \nu_\gamma(dx)
= \frac{1}{\sqrt a}\int_0^\infty \left(\int_0^\infty V(z+x) dz \right) \nu_\gamma(dx)\\
&= \frac{1}{\sqrt a}\int_{\R} \left(\int_{-\infty}^0 V(z-x) dz \right) \nu_\gamma(dx) = \frac{1}{\sqrt a}\int_{-\infty}^0 V\ast \nu_\gamma
= \frac{1}{\sqrt a}\int_{-\infty}^0T\lambda_\gamma,
\end{align*}
where exchanging the order of integration is justified by $V \in L^1 (\R)$ and $\| \nu_\gamma \|_{TV} \leq \gamma$. In conclusion, from \eqref{en:g} and Lemma \ref{lem:cpness:est} we obtain 
\begin{equation*} 
C\geq F(\nu_\gamma) = \frac12 \| \lambda_\gamma \|_2^2  -
\frac{1}{\sqrt{a}}\int_{-\infty}^0 T \lambda_\gamma \geq \frac12 \| \lambda_\gamma \|_2^2  - C ( \| \lambda_\gamma \|_2 + 1 );
\end{equation*}
hence $(\lambda_\gamma)$ is bounded in $L^2( \R )$. Consequently, by \eqref{for:lem:cpness:est:nu}, $\nu_\gamma ([r, s])$ is uniformly bounded in $\gamma$ for any bounded interval $[r, s] \subset [0, \infty)$, with a bound that only depends on $s-r$. Therefore, along a subsequence,
\begin{equation} \label{for:cpness:result}
  \nu_\gamma \weakto \nu \text{ vaguely,} 
  \qquad 
  \text{and } \lambda_\gamma = T\nu_\gamma \weakto \lambda \text{ in } L^2.
\end{equation}

\medskip

\textit{Step 3. Characterisation of the limits $\nu$ and $\lambda$.}
By the basic properties of the vague topology, $\nu \in \mathcal M ([0,\infty))$ satisfies $\nu (dx) \geq - \rhomin (0) dx$ and $\sup_{x \geq 0} \nu^+ ([x, x+1]) < \infty$. 
Consequently, by Lemma \ref{lem:T}.\eqref{lem:T:on:A}, $\nu \in X'$. We claim that $\nu_\gamma$ converges weakly-\textasteriskcentered\ in $X'$ to $\nu$. If so, the continuity of $T$ (Lemma~\ref{lem:T}.\eqref{lem:T:BLSO}) and \eqref{for:cpness:result} imply that $\lambda = T\nu \in L^2(\R{})$, and thus $\nu \in \mathcal A$.

To prove this claim, we take $f \in X$ arbitrary and a large $M > 0$ fixed and write 
\begin{equation*}
\int_0^\infty f \nu_\gamma 
=
\int_0^M f \nu_\gamma + \int_M^\infty f \nu_\gamma.
\end{equation*}
The first term in the right-hand side converges to $\int_0^M f \nu$ as $\gamma\to \infty$ because of the vague convergence of $\nu_\gamma$ to $\nu$. We now show that the second term is infinitesimal for large $M$. Indeed, by Lemma~\ref{lem:g:nu:bounds} and Lemma~\ref{lem:cpness:est}, 
\begin{align}
\int_M^\infty f \nu_\gamma &= \int_0^\infty (f\chi_{(M,\infty)}) \nu_\gamma \leq C\|f\chi_{(M,\infty)}\|_X \sup_{x\geq 0} |\nu_\gamma|([x,x+1])\nonumber\\
&= C\|f\chi_{(M,\infty)}\|_X \sup_{x\geq 0} (\nu^+_\gamma+\nu^-_\gamma)([x,x+1])\nonumber\\
&\leq C\|f\chi_{(M,\infty)}\|_X \big(\|\lambda_\gamma\|_2 + 1 + \rho^\ast(0)\big),\label{est:tail:X}
\end{align}
where $\chi_{(M,\infty)}$ denotes the characteristic function of the interval $(M,\infty)$. Clearly the term in \eqref{est:tail:X} is infinitesimal for large $M$, as claimed.

\bigskip

\textbf{Lower bound.} To prove \eqref{for:thm:first:order:Gconv:liminf}, we observe from \eqref{diff:en:nuc} that it is sufficient to show that
\begin{equation} \label{for:pf:liminf:ineqs}
\liminf_{\gamma \rightarrow \infty} \int_0^\infty h_\gamma \nu_\gamma \geq 0, 
\quad \textrm{and} \quad
\liminf_{\gamma \rightarrow \infty} F (\nu_\gamma) \geq F(\nu).
\end{equation}
We start with the first inequality in \eqref{for:pf:liminf:ineqs}. First of all, if we write $h_\gamma=h_\gamma^+-h_\gamma^-$ and $\nu_\gamma=\nu_\gamma^+-\nu_\gamma^-$, we have 
\begin{align*}
\int_0^\infty h_\gamma \nu_\gamma \geq -\int_0^\infty h_\gamma^+ \nu_\gamma^- -\int_0^\infty h_\gamma^- \nu_\gamma^+,
\end{align*}
the other terms being non-negative. Therefore the first inequality in \eqref{for:pf:liminf:ineqs} follows if we prove that 
\begin{equation}\label{small-claim}
\limsup_{\gamma\to\infty} \int_0^\infty h_\gamma^+\nu_\gamma^- = 0, \quad \textrm{and} \quad \limsup_{\gamma\to\infty} \int_0^\infty h_\gamma^-\nu_\gamma^+ = 0.
\end{equation} 
We note that, since $\supp \nu_\gamma^-\subseteq [0,2\gamma\sqrt{a}]$ and $\nu_\gamma^- (dx) \leq \rhomin(0) dx$, 
\begin{align*}
\int_0^\infty h_\gamma^+\nu_\gamma^- = \int_0^{2\gamma\sqrt{a}} h_\gamma^+\nu_\gamma^- \leq \rhomin(0) \int_0^{2\gamma\sqrt{a}} h_\gamma^+.
\end{align*}
The bounds \eqref{hc1} and \eqref{hc2} on $h_\gamma$ entail
\begin{align}\label{small-claim1}
\int_0^{2\gamma\sqrt{a}} h_\gamma^+ = \int_0^{\sqrt \gamma} h_\gamma^++ \int_{\sqrt{\gamma}}^{2\gamma\sqrt{a}-\sqrt{\gamma}}  h_\gamma^+ + \int_{2\gamma\sqrt{a}-\sqrt{\gamma}}^{2\gamma\sqrt{a}}  h_\gamma^+ \leq \frac{C}{\sqrt \gamma} + C\sigma(\gamma),
\end{align}
which converges to zero as $\gamma\to \infty$. Analogously, since $(\supp h_\gamma^-)\cap [0,\infty)\subseteq [0,2\gamma\sqrt{a})$, by \eqref{hc1}, \eqref{hc2} and Lemma \ref{lem:cpness:est} we have 
\begin{align}\label{small-claim2}
\int_0^\infty h_\gamma^-\nu_\gamma^+ 
&= 
\int_0^{\sqrt \gamma} h_\gamma^-\nu_\gamma^++ \int_{\sqrt{\gamma}}^{2\gamma\sqrt{a}-\sqrt{\gamma}}  h_\gamma^-\nu_\gamma^+ + \int_{2\gamma\sqrt{a}-\sqrt{\gamma}}^{2\gamma\sqrt{a}}  h_\gamma^-\nu_\gamma^+\nonumber\\
& \leq 
\frac{C}{\gamma} \int_{0}^{\sqrt \gamma} \nu_\gamma^+ + C\,\frac{\sigma(\gamma)}{\gamma}\int_{\sqrt \gamma}^{2\gamma\sqrt{a}-\sqrt \gamma} \nu_\gamma^+ 
+ \frac{C}{\gamma} \int_{2\gamma\sqrt{a}-\sqrt \gamma}^{2\gamma\sqrt{a}} \nu_\gamma^+ \nonumber\\
&\leq 
\frac{C}{\sqrt \gamma} + C\sigma(\gamma) \to 0\quad \textrm{as} \quad \gamma\to \infty.
\end{align}
The bounds \eqref{small-claim1} and \eqref{small-claim2} above imply \eqref{small-claim}, and hence the first inequality in \eqref{for:pf:liminf:ineqs}.

\medskip

We now prove the second inequality in \eqref{for:pf:liminf:ineqs}. From \eqref{for:F}, \eqref{for:cpness:result} and Step~3 of the compactness proof, we obtain immediately that
\begin{equation*}
  \liminf_{\gamma \rightarrow \infty} \| T \nu_\gamma \|^2_2 \geq \| T \nu \|^2_2.
\end{equation*}
Therefore, to conclude the proof of the liminf inequality, it is sufficient to show that 
\begin{equation*} 
\int_0^\infty g \nu_\gamma 
\to \int_0^\infty g \nu
\quad \text{as} \quad \gamma \to \infty.
\end{equation*}
Since $g$ is continuous, it follows from the vague convergence of $\nu_\gamma$ to $\nu$ that $\int_0^M g \nu_\gamma \to \int_0^M g \nu$ as $\gamma \to 0$ for any $M > 0$ fixed. The remaining part of the integral is arbitrarily small for $M$ large enough, uniform in $\gamma$, by Lemma~\ref{lem:cpness:est} and Lemma~\ref{lem:g:nu:bounds}.

\bigskip

\textbf{Upper bound.} Let $\nu \in \mathcal{A}$. We first show in Steps 1 and 2 that by the usual density argument it is sufficient to prove \eqref{for:thm:first:order:Gconv:limsup} for $\nu \in L^2(0, \infty)$ with bounded support and $\nu > -\rhomin (0)$. 

\medskip
\noindent
\textit{Step 1. Approximation of $\nu$ in $L^2 (0, \infty)$.}
We leave the proof to Appendix \ref{app:Step:1}.

\medskip
\noindent
\textit{Step 2. Truncation and retraction.} 
For $\nu \in \mathcal{A} \cap L^2 (0, \infty)$, we set 
\[\nu_n := \Big( 1 - \frac1n \Big) \nu\llcorner[0,n].\] 
It is easy to see that $\nu_n \in \mathcal A \cap L^2 (0, \infty)$ satisfies the desired bound $\nu_n > -\rhomin (0)$, and that $\nu_n$ converges both in the vague topology and in $L^2(0, \infty)$ to $\nu$. Convergence of $F(\nu_n)$ to $F(\nu)$ as $n \to \infty$ follows from $L^2$-continuity of $F : L^2 (0, \infty) \to \R$ (by Lemma \ref{lem:T}.\eqref{lem:T:BLSO} and $g \in L^2(0, \infty)$).

\medskip
Next we construct a recovery sequence $(\nu_\gamma)$ for $\nu \in \mathcal A \cap L^2(0, \infty)$ with support in $[0, M]$ and $\nu > -\rhomin (0)$. We define
\begin{equation*} 
\nu_\gamma (x) := \nu (x) - \sigma_\gamma \tilderhomin (x), \qquad \sigma_\gamma := \frac1\gamma \int_0^M \nu.
\end{equation*}
It is clear that $\nu_\gamma$ converges to $\nu$ in both the vague and $L^2$-topology. We conclude that $\nu_\gamma \in \mathcal A_\gamma$, since $\int_0^\infty \nu_\gamma = 0$ and $\nu_\gamma \geq - \tilderhomin$ for sufficiently large $\gamma$.

It remains to prove that $F_\gamma (\nu_\gamma) \rightarrow F(\nu)$ as $\gamma\to \infty$. We use \eqref{diff:en:nuc} to write
\begin{align} \label{for:pf:Fc:rec:seq}
F_\gamma (\nu_\gamma)
 = F (\nu_\gamma) + \int_0^M h_\gamma \nu - \sigma_\gamma \int_0^{2\sqrt a \gamma} h_\gamma \tilderhomin.
\end{align}
By using again continuity of $F$, we obtain that $F (\nu_\gamma) \to F (\nu)$ as $\gamma \to \infty$. With the bound on $h_\gamma$ given by \eqref{hc1} we show that the remaining two terms in the right-hand side of \eqref{for:pf:Fc:rec:seq} are small by the estimates
\begin{equation*}
  \bigg| \int_0^{M} h_\gamma \nu \bigg| 
  \leq \sqrt M \, \frac C\gamma \| \nu \|_2,
  \quad \text{and} \quad
  \bigg| \sigma_\gamma \int_0^{2\sqrt a \gamma} h_\gamma \tilderhomin \bigg|
  \leq \sigma_\gamma (2\sqrt a \gamma) \frac C\gamma \rhomin (0).
\end{equation*}
\end{proof}

\section{Numerics}\label{sec:numerics}
In this section we present some numerical results which illustrate the scaling properties of the boundary layer and the degree of fit or misfit between the predictions of the continuum limit energies $E$ and $F$, and the minimisers of the discrete energy $E_n$. The numerical simulations presented in this section are all obtained for the specific choice of the potential given in \eqref{defV}.

\subsection{Rescaling of the density.}\label{sec:scalings} The density $\nu$ appearing in the boundary layer energy $F_\gamma$ in~\eqref{Fc} is defined by blowing up the pile-up domain by a factor $\gamma$. This behaviour is confirmed by the $\Gamma$-convergence result. We now illustrate some formal calculations and numerical results that suggest a similar scaling for the \emph{discrete} densities.

As predicted by the asymptotics in \cite{Hall11}, Figure \ref{fig:tilde:rho:n2} illustrates that the size of the boundary layer region at the lock is $\mathcal O (1/\sqrt n)$ in the case $\gamma_n = \sqrt n$. In Figure \ref{fig:tilde:rho:n2} we have  computed the dislocation wall positions $x^n_\ast$ by using Newton's method to solve $\nabla E_n (x) = 0$. To define a discrete density in terms of the rescaled coordinates $\gamma_n x^n_\ast$, we mimic the transformation $(\gamma_n)_\to$ (defined for the continuum densities in \eqref{def:scalingOperator}):
\begin{equation} \label{for:tilde:rho:n}
\tilde \rho_n (x) := \rho_n (x / \gamma_n),
\qquad
x \in \left\{ \gamma_n x^n_{\ast, i} : i = 1,\ldots,n-1 \right\}.
\end{equation}

Although Figure \ref{fig:tilde:rho:n2} illustrates the size of the boundary-layer profile, the slopes of the densities at the right-end of the domain are different. This difference is a result of the $n$-independent bulk behaviour of the unscaled densities $\rho_n$, which converge to the bulk profile~$\rhomin$ (cf.~\cite[Thm.~7]{GeersPeerlingsPeletierScardia13}).

For this reason it is more instructive to compare the densities $\nu_n := \tilde \rho_n - \tilderhomin$, which we do in Figure~\ref{figs1011}. 
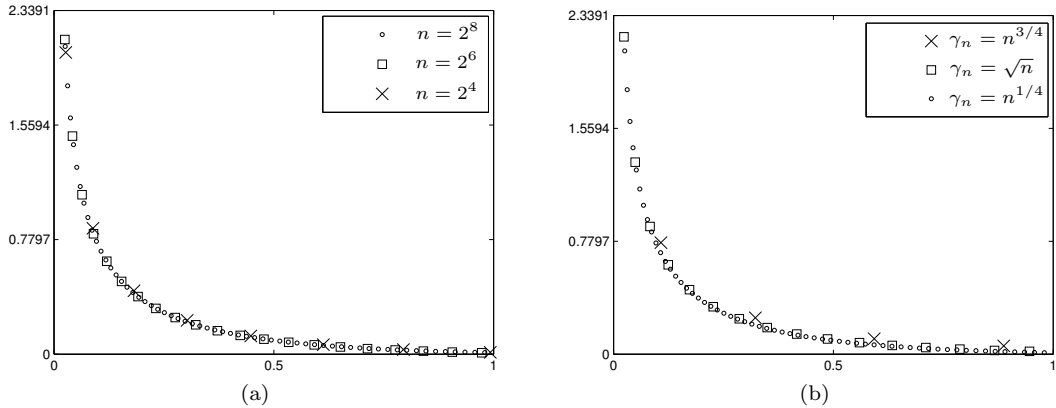
\begin{figure}[h]
\centering
\small
\subfigure[]{%
\labellist
\pinlabel \scriptsize $n=2^8$ at 326 255
\pinlabel \scriptsize $n=2^6$ at 325 233
\pinlabel \scriptsize $n=2^4$ at 325 210
\endlabellist
\includegraphics[width=2.6in]{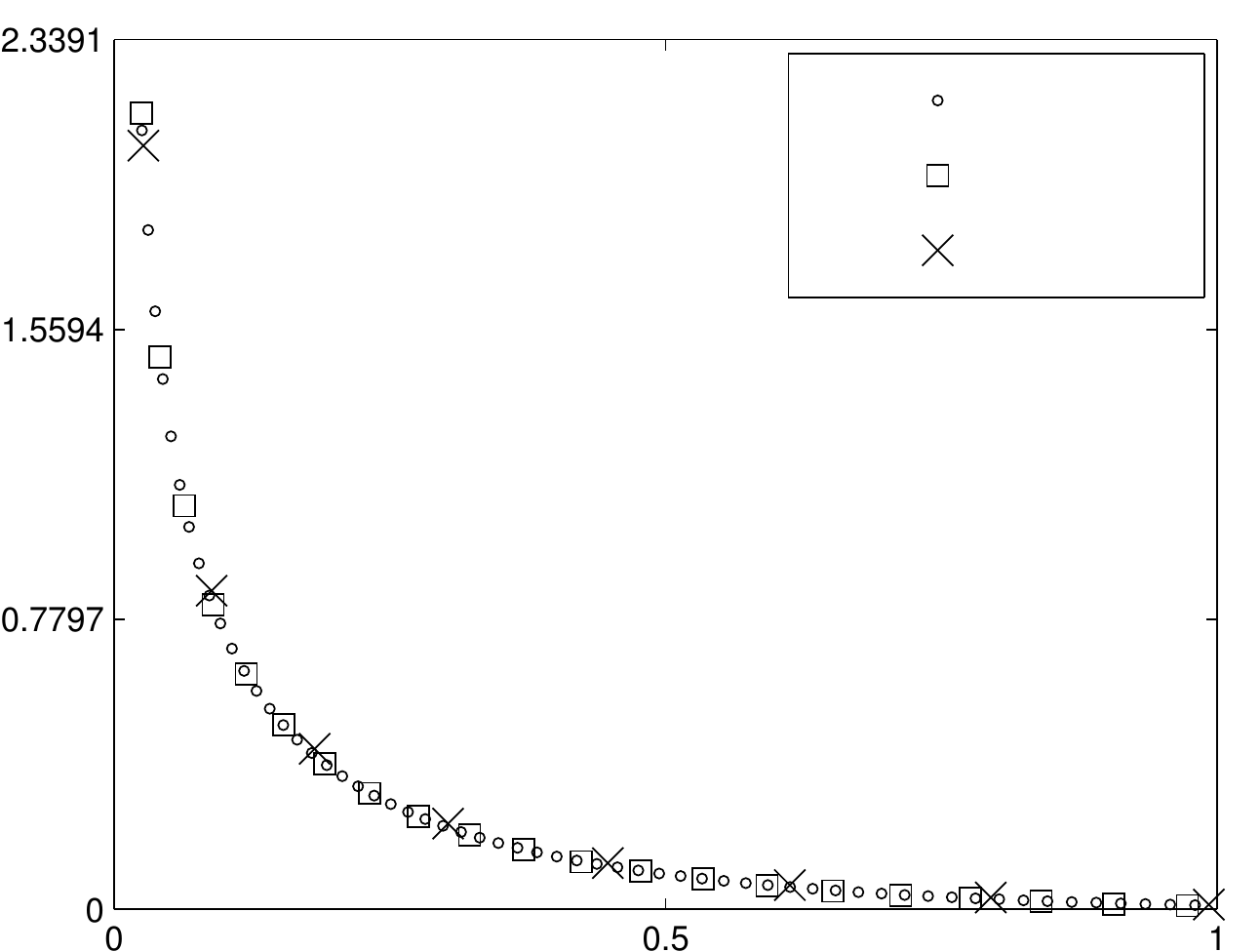}}
\qquad
\subfigure[]{%
\labellist
\pinlabel \scriptsize $\gamma_n=n^{3/4}$ at 318 250
\pinlabel \scriptsize $\gamma_n=\sqrt{n}$ at 314 227
\pinlabel \scriptsize $\gamma_n=n^{1/4}$ at 318 206
\endlabellist
\includegraphics[width=2.6in]{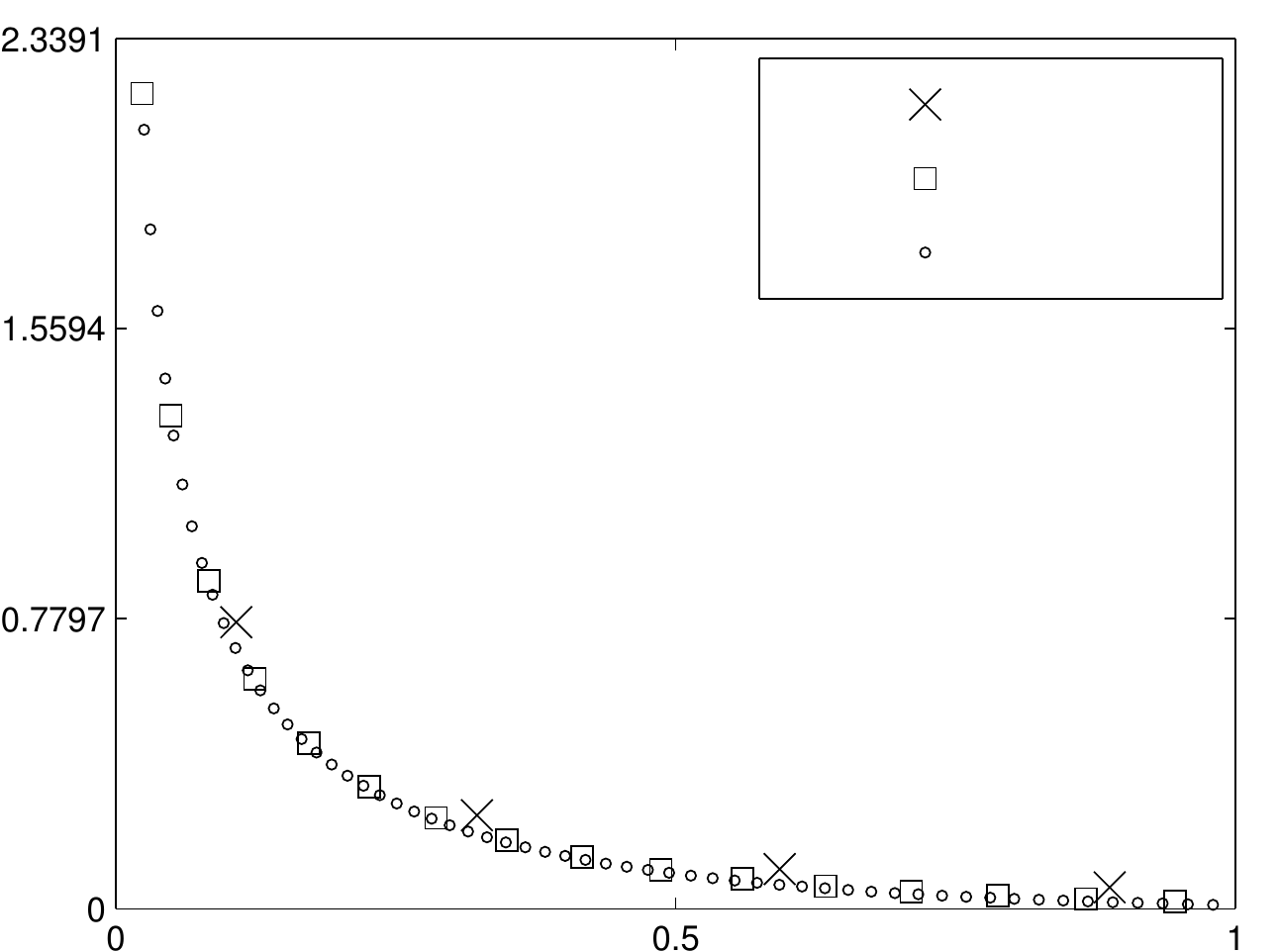}}
\caption{Plots of the rescaled density $\nu_n$ for (a) $\gamma_n=n^{1/4}$ and various values of $n$; (b) $n=2^8$ and various scalings of $\gamma_n$.}
\label{figs1011}
\end{figure}
The good agreement between the profiles of $\nu_n$ for different values of $n$ and different scaling of $\gamma_n$ (conform $1 \ll \gamma_n \ll n$) suggest that also in the discrete case the density must be rescaled as in~\eqref{for:tilde:rho:n}. 

\subsection{Rescaling of the energy}

From the analysis in this paper one could try to infer that also for the discrete energies~$E_n$ in~\eqref{discreteEcn} the right scaling to analyse the boundary layer at the lock is $1/\gamma_n$, namely that $|E_n (x^n_\ast) - E_n(\rhomin^n)|\propto 1/\gamma_n$. Here, $\rhomin^n$ are the discrete dislocation wall positions generated by $\rhomin$, i.e.,
\[
\rhomin^n = \frac1n \sum_{i=1}^n \delta_{y_i^n},
\quad
\text{where } 
\quad 
y_i^n := 2 \sqrt a \left( 1 - \sqrt{ 1 - \frac in } \right) 
\quad
\text{ solves } 
\quad 
\int_0^{y_i^n} \rhomin = \frac in.
\]
We investigate this question by means of numerical computations, for $\gamma_n$ being a power of $n$. Setting $\alpha_n := |E_n (x^n_\ast) - E_n(\rhomin^n)|$, and assuming a power-law relation also for $\alpha_n$, namely $\alpha_n \propto n^{-p}$, we try to determine $p$ by computing
\begin{equation} \label{for:pn}
  p_n := \frac{ \log \alpha_n - \log \alpha_{2n} }{\log 2},
  \qquad
  \text{for } n = 2^3,2^4,\ldots,2^{11},
\end{equation}
for different expressions of $\gamma_n$. The results are shown in Table \ref{tab:pn}. The data in the first two columns support our conjecture that $p_n \approx \log_n \gamma_n $ for sufficiently large $n$, although the third column seems to contradict it. This disagreement could be due to a much slower convergence to the correct value, or to the presence of a second boundary layer at the free end of the pile-up (see Figure \ref{fig:long:range:disc:vs:ct}), whose energy contribution interferes with the contribution of the boundary layer at the lock. At this stage we can only guess as to the reasons for this discrepancy, and we will return to this issue in a future publication.

\begin{table}[h!]                    
\centering                           
\begin{tabular}{cccc} \toprule   
 & \multicolumn{3}{c}{$\log_n \gamma_n$} \\ \cmidrule(r){2-4}                                 
$n$ & $\dfrac14$ & $\dfrac12$ & $\dfrac34$ \\ \midrule                   
$2^3$  & 0.550 & 0.901 & 0.931 \\                                       
$2^4$  & 0.384 & 0.757 & 0.833 \\                                       
$2^5$  & 0.299 & 0.642 & 0.745 \\                                         
$2^6$  & 0.263 & 0.567 & 0.675 \\                                         
$2^7$  & 0.250 & 0.525 & 0.625 \\                                         
$2^8$  & 0.247 & 0.505 & 0.590 \\                                         
$2^9$  & 0.247 & 0.496 & 0.567 \\                                         
$2^{10}$  & 0.247 & 0.494 & 0.552 \\                                         
$2^{11}$  & 0.248 & 0.493 &  \\ \bottomrule
\end{tabular}     
\medskip                   
\caption{Exponents $p$ in \eqref{for:pn} for a power law fit of $|E_n (x^n_\ast) - E_n(\rho^n_\ast)| \propto n^{-p}$, for different values of $n$ and $\gamma_n = n^{1/4}, n^{1/2}, n^{3/4}$. The lower-right value is missing because of memory limitations in MATLAB. }
\label{tab:pn}           
\end{table}   

\subsection{Minimiser of the continuum energy $F$.}\label{Min}

In this section we analyse the minimiser of $F$, namely we consider the minimisation problem:
\begin{equation*}
  \min_{\nu \in \mathcal A} F(\nu),
\end{equation*}
where $F$ and $\mathcal A$ are defined in \eqref{for:F} and \eqref{for:dom:F}. We assume in addition that the minimiser $\nu_\ast$ is in $L^1(0, \infty)$ and satisfies the strict inequality\footnote{Note that the numerically calculated minimisers (Figure~\ref{fig:long:range:disc:vs:ct}) are consistent with this assumption.} $\inf \nu_\ast > -\rhomin(0)$, and derive a first-order necessary condition for minimality for $\nu_\ast$. We consider variations $\nu_\ast + t \mu \in \mathcal A$ for $\mu \in \mathcal M([0, \infty))$ with bounded support and bounded Lebesgue density,  and $|t|$ sufficiently small. Then
\begin{equation*}
0 = \frac d{dt} F (\nu_\ast + t \mu) |_{t=0} = \int_0^\infty (V \ast \nu_\ast) \mu - \int_0^\infty g \mu,
\end{equation*}
and since $\mu$ is arbitrary, this implies that
\begin{equation} \label{for:ELeqn}
V \ast \nu_\ast = g = \frac1{\sqrt a} V \ast \chi_{(-\infty, 0]}
\qquad
\text{a.e. on } (0, \infty).
\end{equation}

To approximate the solution $\nu_\ast$ to \eqref{for:ELeqn} numerically, we follow the method in Section~10.5 of \cite{Heath01}. We approximate $\nu_\ast$ by a step function of the form 
\begin{equation} \label{for:apx:nu}
\nu_\ast \approx \sum_{i=1}^N \lambda_i \chi_{I_i},
\end{equation}
where $\lambda_i$ are the unknowns, while $N$ and the sets $I_i$ are chosen beforehand. Figure \ref{figs1011} suggests to take small intervals $I_i = (a_{i-1}, a_i)$ close to $0$, and larger intervals in the region $x>1$. We choose the intervals $I_i$ so that $|I_i| = C b^i$, and then tune $C$, $b$ and $N$ such that $a_1 = 3 \cdot 10^{-5}$, $200 \in I_N$, and the interval $I_j$ containing the element $1$ has length $|I_j| = 0.1$. This results in $C = 2.727 \cdot 10^{-5}$, $b = 1.1$, and $N = 141$.

After having fixed $a_i$ and $N$, we substitute the approximation \eqref{for:apx:nu} in the Euler-Lagrange equation \eqref{for:ELeqn}, and we solve for $\lambda_i$ the linear system
\begin{equation} \label{for:lin:sys:lambdai}
\sum_{i=1}^N \lambda_i (V \ast \chi_{I_i})(y_j) = \frac1{\sqrt a} (V \ast \chi_{(-\infty, 0]}) (y_j)
\qquad
\text{for } y_j = \frac{a_j - a_{j-1}}2, \: j = 1,\ldots,N.
\end{equation}
This amounts to an asymmetric Galerkin (collocation) approximation, in which the solution space is approximated by piecewise constant functions, and the test space by sums of delta functions supported on the midpoint of each interval. 

Note that we can identify $(V \ast \chi_{I_i})(y_j) = \int_{x_j - a_i}^{x_j - a_{i-1}} V$, and that we have explicitly
\[
\int_y^\infty V = \text{Li}_2 (e^{-2y}) - y \ln (1 - e^{-2 y}), 
\]
where $\text{Li}_2$ is the polylogarithm,
\[
\text{Li}_s(z) = \sum_{k=1}^\infty \frac{z^k}{k^s}.
\]
The solution to the linear system \eqref{for:lin:sys:lambdai} results in the profile of $\nu_\ast$ \eqref{for:apx:nu} as depicted in Figure \ref{fig:nuast:alone}. 

\begin{figure}[h]
\labellist
\pinlabel $\nu_\ast (y_i)$ at 425 520
\pinlabel $\nu_\ast$ at 430 493
\endlabellist
\begin{center}
\includegraphics[width=3in]{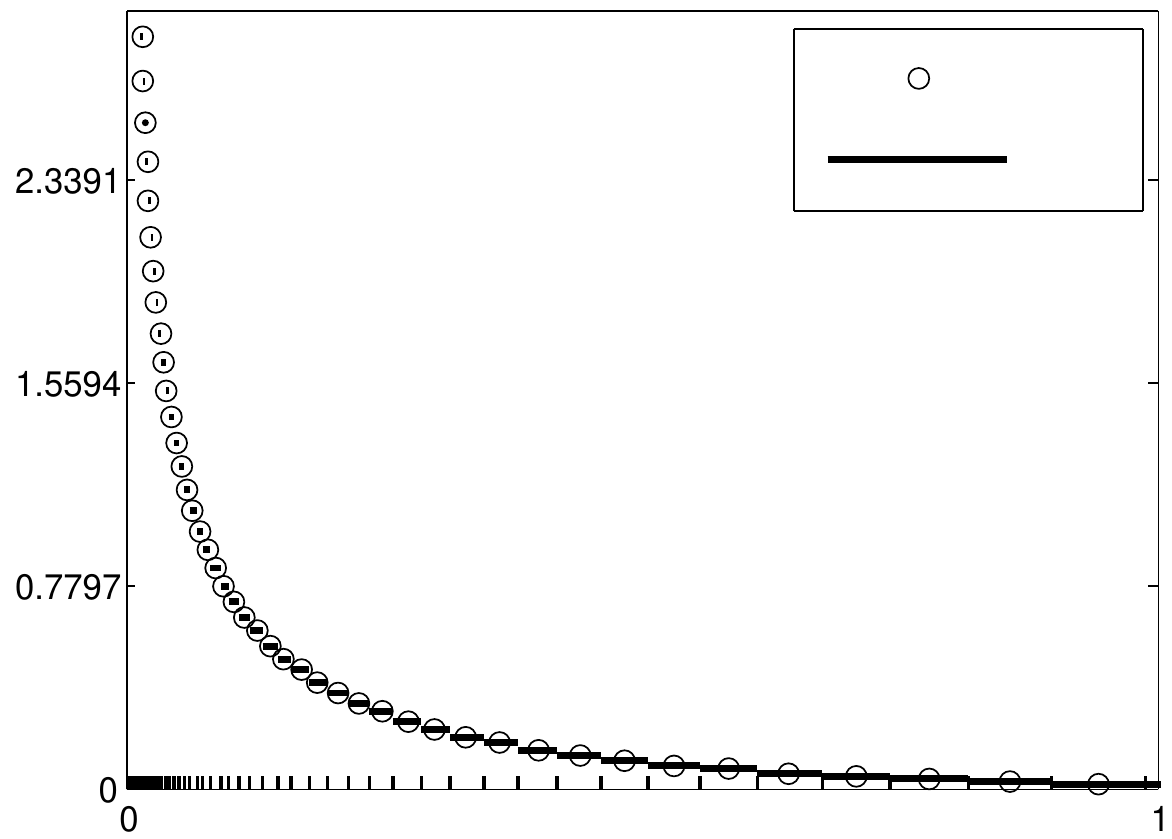}
\end{center}
\caption{Numerical approximation of $\nu_\ast$ \eqref{for:apx:nu}. The intervals $I_i$ are depicted on the $x$-axis. The circles indicate the coordinates $(y_i, \nu_\ast (y_i))$.}
\label{fig:nuast:alone}
\end{figure}

In Figure \ref{fig:BL:disc:vs:ct} we compare $\nu_\ast$ to the discrete density $\nu_n=\tilde\rho_n-\tilderhomin$ constructed from the minimiser $\rho_n$ of $E_n$. In order to make the figure cleaner, instead of plotting the (mathematically correct) piecewise-constant version of Figure~\ref{fig:nuast:alone}, we plot the affine interpolation of the coordinates $(y_i, \nu_\ast (y_i))$.  
\begin{figure}[h]
\labellist
\pinlabel \small $\nu_n$ at 335 242.5
\pinlabel \small $\nu_\ast$ at 335 220
\endlabellist
\begin{center}
\includegraphics[width=3in]{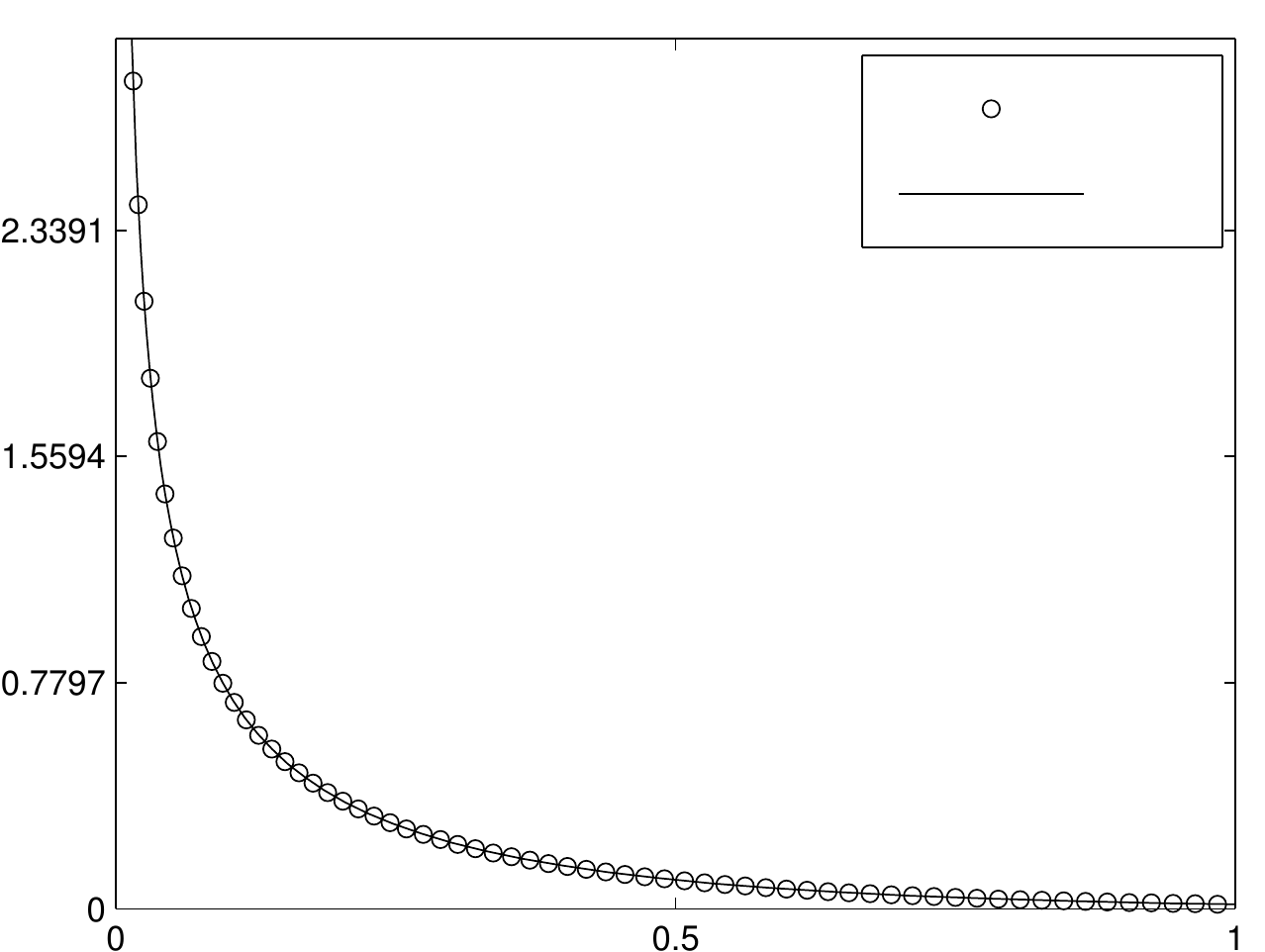}
\end{center}
\caption{Comparison between the discrete density $\nu_n$ for $n = 2^{12}$ and $\gamma_n = \sqrt n$ and the minimiser $\nu_\ast$ of $F$. Only the interval $(0,1)$ is shown, which corresponds to the region $(0,1/\gamma_n)$ for the unscaled densities.}
\label{fig:BL:disc:vs:ct}
\end{figure}
The figure shows a very good agreement between $\nu_n$ and $\nu_\ast$ in the boundary layer region. As we have seen from Figure \ref{figs1011}, the profile of $\nu_n$ depends neither on $n$ nor on the scaling of $\gamma_n$. Therefore, together with Figure \ref{fig:BL:disc:vs:ct}, we conclude that $\nu_\ast$ is the boundary-layer profile of the dislocation density at the lock $x=0$.

\medskip
On the other hand, the agreement between $\nu_n$ and $\nu_\ast$ is not very good outside of the boundary layer region, as shown in Figure \ref{fig:long:range:disc:vs:ct}, where we plot the densities on $(1,2\gamma_n\sqrt{a})$ on a logarithmic scale (remember that $(0,2\gamma_n\sqrt{a})$ is the whole of the pile-up domain, after blowing up the variables by~$\gamma_n$). 
In Section~\ref{sect:firstorder} we proved that $F_\gamma$ $\Gamma$-converges to $F$ with respect to vague convergence, hence minimisers of the energies $F_\gamma$ converge to $\nu_\ast$, minimiser of the limit energy $F$, only on bounded domains. It is therefore not surprising that the densities $\nu_n$, the discrete analogue of the minimiser of $F_\gamma$, do not agree with $\nu_\ast$ on the whole pile-up domain, which is of order $\gamma_n$ in the blown-up variables, and hence growing in size with $n$. 

From Figure \ref{fig:long:range:disc:vs:ct} we can actually infer that there is another `continuous' boundary layer at the right-end of the pile-up domain, since there is a large number of particles deviating from the bulk distribution $\tilderhomin$ (i.e., with density $\nu_n$ deviating from $0$). We expect the profile of this boundary layer to resemble the bulk behaviour of the optimal discrete density of the energy \eqref{discreteEcn} for $\gamma_n=n$. We refer to \cite[Theorem 8]{GeersPeerlingsPeletierScardia13} and to Figure~6 in the same paper for the analysis of that case. The analysis of the boundary layer at the free end of the pile-up region is however beyond the scope of this paper.

\medskip
Finally, as noted in the introduction, $\nu_n$ (or $\nu_\ast$) curiously attains small negative values before matching the line $y=0$. As this `dip' is present in the numerical approximations of both $\nu_n$ and $\nu_\ast$, it does not appear to be  a numerical artifact. We remain with questions regarding how this `dip' relates to the choice of $V$. Does every convex decreasing $V$ with finite first moment have such a `dip'? Are there choices of $V$ for which $\nu_\ast$ is decreasing? Does $V$ have more than one `dip' (the numerical solutions for $\nu_n$ and $\nu_\ast$ are too coarse to give a speculative answer), and how does the answer depend on the choice of $V$? The answers to these questions are  beyond the scope of this paper. 
\medskip

\bigskip

\noindent
\textbf{Acknowledgements.}
The work of MAP and PvM is supported by NWO Complexity grant 645.000.012 and NWO VICI grant 639.033.008. 
LS acknowledges the support of The Carnegie Trust. All authors are grateful to Cameron Hall, John Ockendon and Jon Chapman for various interesting discussions.
\bigskip

\appendix

\section{Properties of the operator $T$}
\label{app:T}

Here we prove Lemma \ref{lem:T} in more generality through Lemma \ref{lem:T:app}. The generality comes for free after introducing the Hilbert spaces $X_{k,j}$ in \eqref{for:defn:Xkj}, which are naturally related to the Fourier transform. 

First, we recall in Lemma \ref{lem:F} several basic properties of the Fourier transform and of the complex Hilbert spaces defined by
\begin{equation*} 
\left. \begin{aligned}
  X_{k,j} (\C) := \Big\{ f \in H^k(\R; \C) : x^{j} f(x) \in L^2(\R; \C) \Big\} & \\
  ( f, \phi )_{X_{k,j}} := \sum_{\ell = 0}^k \int_\R f^{(\ell)} \overline{\phi^{(\ell)}} + \int_\R x^{2j} f (x) \overline{\phi (x)} \, dx&
\end{aligned} \right\} \quad \text{for all } k,j \in \N{}.
\end{equation*}
In case the functions are real-valued, we set
\begin{equation} \label{for:defn:Xkj}
  X_{k,j} := \Big\{ f \in H^k(\R) : x^{j} f(x) \in L^2(\R) \Big\}.
\end{equation}
In particular, we note that $X$ defined in \eqref{for:defn:X} equals $X_{1,2}$, and that $L^2 (\R) = X_{0,0} \supset X_{k,j} \supset \mathcal S (\R)$ for all $k,j \in \N$, where $\mathcal S (\R)$ is the space of Schwarz functions. The analogous statement for complex-valued functions holds as well.

\begin{lem}[Basic properties of $\mathcal F$ and $X_{k,j} (\C)$] \label{lem:F}\noindent
\begin{enumerate}[(i)]
  \item \label{lem:F:i} $\displaystyle \mathcal F \big( x \mapsto x^k f^{(\ell)} (x) \big)
  = i^{\ell + k} (2 \pi)^{\ell - k} \Big( \omega \mapsto \frac{d^k}{d \omega^k} \omega^\ell \widehat f (\omega) \Big)
  \qquad \text{for all } k, \ell \in \N$;
  \smallskip
  \item \label{lem:F:iplus} If $\mu \in \mathcal P (\R)$, then $\mathcal F \mu$ is a positive definite, bounded, uniformly continuous function; 
  \item \label{lem:F:ii} $ \mathcal F $ is a symmetric, invertible, bounded linear operator between $X_{k,j} (\C)$ and $X_{j,k} (\C)$ for all $k,j \in \N{}$;
  \item \label{lem:F:iii} For $k \in \N{}_+$, $X_{k,k} (\C)$ and $X_{k,k}$ are Banach algebras with respect to both the pointwise product and convolution;
  \item \label{lem:F:iv} Let $k \in \N{}_+$, $f \in X_{k,k} (\C)$ and $\xi \in X_{k,k} (\C)'$. Then $\mathcal F ( f \ast \xi ) = \widehat f \,\widehat \xi$ in $X_{k,k} (\C)'$. 
 \end{enumerate}
\end{lem}

\begin{lem} \label{lem:T:app}
The linear operator $T$ defined by \eqref{operator:T:L2} satisfies the following properties.
\begin{enumerate}[(i)]
  \item \label{lem:T:app:real} $T$ maps $L^2 (\R)$ to itself;
  \item \label{lem:T:app:BLSO} For all $k \in \N$ and $j = 0,1,2$, $T$ is a symmetric bounded linear operator that maps $X_{k,j}$ to itself, and $p \mathcal I := (\varphi \mapsto p \varphi)$ is a symmetric bounded linear operator that maps $X_{j,k}$ to itself;
  \item \label{lem:T:app:extension} For all $k \in \N$ and $j = 0,1,2$, the operator $p \mathcal I$ is symmetric, bounded and linear, and maps $X_{j,k}'$ to itself. Furthermore, the extension of $T$ to $X_{k,j}'$ given by
\begin{equation*} 
\text{for } \xi\in X_{k,j}', \ f \in X_{k,j}, \qquad
\langle T\xi,f\rangle 
:= \langle \xi,Tf\rangle 
= \big\langle \mathcal F^{-1} ( p \widehat \xi ), f \big\rangle
\end{equation*}
is well-defined as a symmetric bounded linear operator mapping $X_{k,j}'$ to itself;
  \item \label{lem:T:app:T2:equals:Vast} For all $\xi \in X_{2,2}'$, it holds that $T^2 \xi = V \ast \xi$, and $\displaystyle \| T \xi \|_2^2 = \int \widehat V | \widehat \xi |^2$ as elements of $[0, \infty]$;
  \item \label{lem:T:app:Tnu2:equals:Vnunu} For all $\nu \in \mathcal M ([0,\infty))$ for which there exists $f \in X_{1,1}$ such that $\nu + f \geq 0$ is a finite measure, it holds that 
  \begin{equation*}
    \int_\R (T \nu)^2(x) dx = \int_0^\infty (V \ast \nu)(x) \nu (dx).
  \end{equation*}
\end{enumerate}
\end{lem}

\begin{proof}[Proof of Lemma \ref{lem:T:app}] We now prove the claims \textit{(i)}--\textit{(v)}.

 \textit{(i)}. Let $f \in L^2 (\R)$ be arbitrary. Since $p \in W^{2, \infty} (\R) \subset L^\infty(\R)$, it holds that $p \widehat f \in L^2(\R; \C)$, and thus $T f \in L^2(\R; \C)$. Moreover, since $\widehat f (-\omega) = \overline{\widehat f (\omega)}$ and $p$ is real-valued and even, we also have that $(p \widehat f) (-\omega) = \overline{(p \widehat f) (\omega)}$. Hence, $Tf = \mathcal F^{-1} \big( p\widehat f \big)$ is real-valued.
\smallskip

 \textit{(ii)}. Since $T = \mathcal F^{-1} (p \mathcal I) \mathcal F$, Lemma \ref{lem:F}.\eqref{lem:F:ii} implies that it suffices to prove that $p \mathcal I := (\varphi \mapsto p \varphi)$ is a symmetric bounded linear operator mapping $X_{j,k}$ to itself. Symmetry is obvious. We prove boundedness by showing the existence of some $C > 0$ such that $\| p f \|_{j,k} \leq C \| f \|_{j,k}$ for all $f \in X_{j,k}$. We obtain such a constant $C$ from the following computations, which rely on $p$ belonging to $W^{2, \infty} (\R)$: for $0\leq \ell \leq j\in \{0,1,2\}$ we have
\begin{equation*}
      \left\| ( p f )^{(\ell)} \right\|_2
      \leq \sum_{i = 0}^\ell \bin \ell i \left\| p^{(i)} f^{(\ell-i)} \right\|_2 
      \leq C \sum_{i = 0}^\ell \big\| p^{(i)} \big\|_\infty \left\| f^{(\ell-i)} \right\|_2 
      \leq C \left\| p \right\|_{W^{2, \infty}} \left\| f \right\|_{k,j};
\end{equation*}
moreover
\begin{equation*}
      \| x^k p f \|_2
      \leq \|p\|_\infty \| x^k f \|_2
      \leq \|p\|_\infty \| f \|_{j,k}.
\end{equation*}

\smallskip

 \textit{(iii)}. It follows directly from the previous step that $p \mathcal I$ and $T$ are symmetric bounded linear operators on $X_{j,k}'$ and $X_{k,j}'$ respectively. Since $p$ is an even function, we obtain the representation formula for $T \xi$ with $\xi \in X_{k,j}'$ from
\begin{equation*}
  \langle T \xi, f \rangle
  = \langle \xi, T f \rangle
  = \big\langle \xi, \mathcal F^{-1} \big( p \widehat f \big) \big\rangle
  = \langle \widecheck \xi, p \widehat f \rangle
  = \langle \widehat \xi, p \widecheck f \rangle
  = \langle p \widehat \xi, \widecheck f \rangle
  = \big\langle \mathcal F^{-1} \big( p \widehat \xi \big), f \big\rangle,
  \quad \text{for all } f \in X_{k,j}.
\end{equation*}

 \textit{(iv)}. By (V2) we have that $V \in L^1 (\R) \subset X_{1,1}'$. Then, Lemma \ref{lem:F}.\eqref{lem:F:iv} implies that $T^2 f = V \ast f$ for all $f \in X_{k,k}$ for any $k \in \N_+$. By using the extension of $T$ in the previous step, which requires $k \leq 2$, we obtain $T^2 \xi = V \ast \xi$ for all $\xi \in X_{2,2}'$.

Next we prove that $\| T \xi \|_2^2 = \int \widehat V | \widehat \xi |^2$ as elements of $[0, \infty]$. To avoid confusion, we introduce a different notation for the extension of the $L^2$-norm to linear functionals. For any $\eta \in X_{2,2}'$, we define
\begin{equation*} 
  \| \eta \|_{L^2(\R)'} 
  := \sup_{f \in X_{2,2}} \frac{ \langle \eta, f \rangle }{ \| f \|_2 }
  \in [0, +\infty],
\end{equation*}
which is consistent with the usual definition of $\| \cdot \|_{L^2(\R)'}$ (in which the supremum is taken instead over all $f \in L^2(\R)$) because $X_{2,2}$ is dense in $L^2(\R)$. Likewise, we interpret $\int \widehat V | \widehat \xi |^2$ as $\| p \widehat \xi \|_{L^2(\R)'}^2$. The assertion follows from step~\textit{(iii)} by
\begin{equation*}
  \| p \widehat \xi \|_{L^2(\R)'}
  = \| \mathcal F T \xi \|_{L^2(\R)'}
  = \| T \xi \|_{L^2(\R)'},
\end{equation*}
where the Fourier transform does not change the value of the $L^2(\R)'$-norm. 
\smallskip

\textit{(v)}. It is not clear how $\int (V \ast \nu) \nu$ is defined without using the asserted structure that $f \in X_{1,1}$ and $\mu := \nu + f \geq 0$ is a finite measure. Therefore, we define
  \begin{align} \label{lem:T:pf:3}
    \int (V \ast \nu) \nu
    = \int \big( V \ast (\mu - f) \big) (\mu - f)
    := \int (V \ast \mu) \mu - 2 \int (V \ast f) \mu + \int (V \ast f) f.
  \end{align}
We claim that the right-hand side is well-defined in $(-\infty, \infty]$. Indeed, since $f \in X_{1,1}$, we have that $V \ast f = T^2 f \in X_{1,1} \subset C_b (\R)$. Moreover, since $\mu$ is a finite measure, the second and third integrals are finite. Since $V$ and $\mu$ are non-negative, we find that $V \ast \mu \geq 0$, and thus the first integral is defined as an element of $[0, \infty]$. Moreover, since $\| V \ast \mu \|_1 \leq \|V\|_1 \|\mu\|_{TV} < \infty$, Fubini applies to $\int (V \ast f) \mu = \int (V \ast \mu) f$, which shows that the definition in \eqref{lem:T:pf:3} is consistent with the case in which $\int (V \ast \mu) \mu < \infty$. 

As in step \textit{(iv)}, we interpret $\int (T \nu)^2$ as $\| T \nu \|_{L^2(\R)'}^2 \in [0, \infty]$. Here we rewrite it in terms of $f$ and $\mu$. Whenever convenient, we regard $\mu$ as a bounded linear operator on $C_b (\R)$, defined by $\langle \mu, f \rangle := \int_\R f \, d\mu$. Since $C_b (\R) \supset X_{1,1}$, it follows that $\mu \in C_b (\R)' \subset X_{1,1}'$, and thus $T \mu \in X_{1,1}'$ is well-defined by step \textit{(iii)}.

From $T\nu = T \mu + (Tf, \, \cdot \, )_2$ we obtain that
$T \nu \in L^2(\R)'
  \Longleftrightarrow
  T \mu \in L^2(\R)'$.
If $T \nu \in L^2(\R)'$, then we set $f_\nu$ and $f_\mu$ as the Riesz representatives in $L^2(\R)$ of $T \nu$ and $T \mu$ respectively. We obtain
\begin{equation} \label{lem:T:pf:2}
  \| T \nu \|_{L^2(\R)'}^2
  = \| f_\nu \|_2^2
  = \| f_\mu \|_2^2 - 2 (f_\mu, Tf)_2 + \| Tf \|_2^2
  = \| T \mu \|_{L^2(\R)'}^2 - 2 \langle T \mu, Tf \rangle + \int (V \ast f) f.
\end{equation}
If $T \nu \notin L^2(\R)'$, then we claim that the right-hand side in \eqref{lem:T:pf:2} equals $+\infty$. Indeed, $T \nu \notin L^2(\R)'$ implies $T \mu \notin L^2(\R)'$ and thus $\| T \mu \|_{L^2(\R)'} = + \infty$. It remains to check that $| \langle T \mu, Tf \rangle | < \infty$. Since $\mu \in X_{1,1}'$ and $f \in X_{1,1}$, we have that $T\mu \in X_{1,1}'$ and $Tf \in X_{1,1}$, and thus $\langle T \mu, Tf \rangle$ is indeed finite. In conclusion, \eqref{lem:T:pf:2} holds for any $\nu$ as in the assumption.

It remains to prove that the right-hand sides of \eqref{lem:T:pf:3} and \eqref{lem:T:pf:2} are equal. The equality of the second terms in both expressions follows from 
\begin{equation*}
  \langle T \mu, Tf \rangle
  = \langle \mu, T^2 f \rangle
  = \int (V \ast f) \mu,
\end{equation*}
and thus it remain to show that $\int (V \ast \mu) \mu = \| T \mu \|_{L^2(\R)'}^2$ in $[0, \infty]$. We prove this statement by establishing 
the following results:
\begin{subequations}
\label{lem:T:pf:4}
\begin{alignat}4
\label{lem:T:pf:4:a}
&\int (V \ast \mu) \mu & \,< \infty 
&&\quad \Longrightarrow \quad &\| T \mu \|_{L^2(\R)'}^2 &\leq \int (V \ast \mu) \mu, \\
&\| T \mu \|_{L^2(\R)'} &< \infty
&&\quad \Longrightarrow  \quad &\int (V \ast \mu) \mu \, &\leq \| T \mu \|_{L^2(\R)'}^2.
\label{lem:T:pf:4:b}
\end{alignat}
\end{subequations}

We first prove \eqref{lem:T:pf:4:a}. Given any $n \in \N$, let $V_n$ be as in (V4). Since $\widehat{V_n}, \widehat V \in L^\infty (\R)$ are non-negative, they define (semi-)inner products on $L^2 (\R)$ by
\begin{equation*}
  (f, \varphi)_V := \int (V \ast f) \varphi = \int \widehat V \widehat f \widehat \varphi,
  \quad \text{and} \quad 
  (f, \varphi)_{V_n} := \int (V_n \ast f) \varphi = \int \widehat{ V_n } \widehat f \widehat \varphi.  
\end{equation*} 
We construct Hilbert spaces as the closure of $L^2 (\R)$ with respect to the corresponding norms $\| \cdot \|_V$ and $\| \cdot \|_{V_n}$:
\begin{equation*}
  H_V := \overline{ L^2 (\R) }^{ \| \cdot \|_V },
  \quad \text{and} \quad 
  H_{V_n} := \overline{ L^2 (\R) }^{ \| \cdot \|_{V_n} }.
\end{equation*}
Since it is given that $\int (V \ast \mu) \mu < \infty$, we obtain $\mu \in H_V$. Because $V_n \leq V$ and $\mu \geq 0$, it holds that $\int (V_n \ast \mu) \mu < \int (V \ast \mu) \mu$, and thus $\mu \in H_{V_n}$. Therefore, we obtain
\begin{equation*}
  \int (V_n \ast \mu) \mu 
  = \sup_{\varphi \in X_{1,1}} \bigg[ 2 \int (V_n \ast \varphi) \mu - \int (V_n \ast \varphi) \varphi \bigg].
\end{equation*}
Next, we fix $\varphi \in X_{1,1}$, and use $V_n \to V$ in $L^1(\R)$ and $V_n \uparrow V$ pointwise together with the Monotone Convergence Theorem to pass to the limit $n \to \infty$ in the previous expression, to get
\begin{multline*}
  \int (V \ast \mu) \mu 
  = \lim_{n \to \infty} \int (V_n \ast \mu) \mu
  \geq \lim_{n \to \infty} \bigg[ 2 \int (V_n \ast \varphi) \mu - \int (V_n \ast \varphi) \varphi \bigg] \\
  = 2 \int (V \ast \varphi) \mu - \int (V \ast \varphi) \varphi
  = 2 \langle T \mu, T \varphi \rangle - \int (T \varphi)^2.
\end{multline*}
Taking the supremum over $\varphi \in X_{1,1}$, we obtain 
\begin{equation*}
  \int (V \ast \mu) \mu 
  \geq \sup_{\phi \in T X_{1,1}} \bigg[ 2 \langle T \mu, \phi \rangle - \int \phi^2 \bigg],
\end{equation*}
where $T X_{1,1} := \{T f : f \in X_{1,1} \}$. It follows from $p > 0$ that $T X_{1,1}$ is dense in $X_{1,1}$. Therefore,
\begin{equation*}
  \int (V \ast \mu) \mu 
  \geq \sup_{\varphi \in X_{1,1}} \bigg[ 2 \langle T \mu, \varphi \rangle - \int \varphi^2 \bigg]
  = \| T \mu \|_{L^2(\R)'}^2.
\end{equation*}

Next, we prove \eqref{lem:T:pf:4:b}. Since it is given that $L^2(\R)' \ni T\mu = \mathcal{F}^{-1}(p \widehat \mu)$, it follows from $\widehat \mu \in C_b (\R; \C)$ (by Lemma \ref{lem:F}.\eqref{lem:F:iplus}) that $p \widehat \mu$ is a measurable function, and thus we can write 
\begin{equation} \label{lem:T:pf:5}
  \| T \mu \|_{L^2(\R)'}^2 = \int |p \widehat \mu|^2 = \int \widehat V |\widehat \mu|^2.
\end{equation}

We introduce two regularisations: we regularise $V$ again by $V_n$, and we set $\mu_\varepsilon := \widehat{G_\varepsilon} \ast \mu \in L^1 (\R) \cap C_b (\R)$, where $G_\varepsilon$ is the scaled Gaussian satisfying
  \begin{equation*} 
    G_\varepsilon (x) := \exp(-\pi \varepsilon x^2) \weakto 1 \text{ vaguely as } \varepsilon \to 0,
    \quad \text{and} \quad 
    \widehat{G_\varepsilon} = \varepsilon^{-1/2} G_{ 1 / \varepsilon } \weakto \delta_0 \text{ as } \varepsilon \to 0.
  \end{equation*} 
First, we use $V_n \in C_b(\R)$ in combination with $\mu_\varepsilon \otimes \mu_\varepsilon \weakto \mu \otimes \mu$ as $\varepsilon \to 0$ to deduce that 
\begin{equation*}
  \int (V_n \ast \mu_\varepsilon) \mu_\varepsilon
  = \iint_{\R^2} V_n (x - y) \, d (\mu_\varepsilon \otimes \mu_\varepsilon)(x,y)
  \xrightarrow{ \varepsilon \to 0 } \iint_{\R^2} V_n (x - y) \, d (\mu \otimes \mu)(x,y)
  = \int (V_n \ast \mu) \mu.
\end{equation*}
Then, we use that $\mu_\varepsilon \geq 0$ and $V_n \leq V$ to estimate
\begin{equation*}
  \int (V_n \ast \mu_\varepsilon) \mu_\varepsilon
  \leq \int (V \ast \mu_\varepsilon) \mu_\varepsilon
  = \int \widehat V | \widehat{ \mu_\varepsilon } |^2
  = \int \widehat V G_\varepsilon^2 | \widehat{ \mu } |^2.
\end{equation*}
Since $G_\varepsilon^2 \uparrow 1$ pointwise on $\R$, we obtain from the Monotone Convergence Theorem and \eqref{lem:T:pf:5} that
\begin{equation*}
  \int \widehat V G_\varepsilon^2 | \widehat{ \mu } |^2
  \xrightarrow{ \varepsilon \to 0 } \int \widehat V | \widehat{ \mu } |^2
  = \| T \mu \|_{L^2(\R)'}^2.
\end{equation*}
Summarizing our arguments, we have
\begin{equation*}
  \int (V_n \ast \mu) \mu 
  = \lim_{\varepsilon \to 0} \int (V_n \ast \mu_\varepsilon) \mu_\varepsilon
  \leq \| T \mu \|_{L^2(\R)'}^2.
\end{equation*}
We conclude by taking the supremum over all $n$.
\end{proof}

\bigskip

\begin{rem}[Assumptions on $V$ needed for Lemma \ref{lem:T:app}]
For $T$ to be a bounded operator from $X$ to itself (cf.~Lemma \ref{lem:T:app}.\eqref{lem:T:app:BLSO}), more regularity than $p \in L^\infty (\R)$ is required. In the proof we showed that $p \in W^{2, \infty} (\R)$ is sufficient.

The proof of Lemma \ref{lem:T:app}.\eqref{lem:T:app:Tnu2:equals:Vnunu} is technical, and we need (V4) to complete it. Instead of the requirement $p > 0$, it is sufficient to have $|\{p = 0\}| = 0$. However, we need the requirement $p > 0$ elsewhere (for constructing a recovery sequence) to control the tail of $\nu$.
\end{rem}

\section{Technical step in the proof of \eqref{for:thm:first:order:Gconv:limsup}}
\label{app:Step:1}

Here we prove Step 1 of the proof of \eqref{for:thm:first:order:Gconv:limsup}, i.e., we show that it is sufficient to prove \eqref{for:thm:first:order:Gconv:limsup} for $\nu \in \mathcal A \cap L^2 (0, \infty)$.

First we show by density arguments that we can restrict ourselves to those $\nu \in \mathcal A$ with support away from $0$. For any $\nu \in \mathcal A$, we consider $\nu_c := \tau^c_\# \nu$ for any $c>0$, where $\tau^c (x) := x + c$ acts as a horizontal shift. It is clear that $\nu_c$ converges vaguely to $\nu$ as $c \to 0$, and that $\nu_c$ satisfies the upper and lower bounds in the definition of $\mathcal A$. Since by \eqref{operator:T:L2} $T$ commutes with~$\tau^c_\#$, it follows that $ \| T \nu_c \|_2 = \| T \nu \|_2<\infty$. Hence, $\nu_c \in \mathcal A$, and to complete the density argument it remains to show that $\limsup_{c \to 0} \int g \nu_c = \int g \nu$. By splitting the integrals as follows
\begin{equation*}
  \int g \nu - \int g \nu_c 
  = \int_0^M \big[ g(x) - g(x + c) \big] \nu + \int_M^\infty \big[ g(x) - g(x + c) \big] \nu,
\end{equation*}
it is clear that for any $M > 0$ the first term in the right-hand side converges to zero, because $g$ is continuous, and the second term is small in $M$ by Lemma \ref{lem:g:nu:bounds} and Lemma \ref{lem:cpness:est}, by a similar argument as in Step 3 in the proof of Theorem \ref{thm:first:order:Gamma:conv}.

Next, we take $\nu \in \mathcal A$ such that $\supp \nu \subset [c, \infty)$ for some $c > 0$. We consider the approximation $\tilde \nu_\varepsilon := \psi_\varepsilon \ast \nu$, where $\psi_\varepsilon$ is the rescaling of a carefully chosen mollifier $\psi$ (constructed below), and its restriction $\nu_\varepsilon := \tilde \nu_\varepsilon \llcorner[0,\infty)$. We choose $\psi$ as the Fourier transform of 
\begin{equation*}
  \widehat \psi (\omega) 
  := \left\{ \begin{aligned}
    &(1 - \alpha |\omega|)^3 \big( 3 \alpha |\omega| + 1 \big),
    && \text{if } |\omega| \leq 1/\alpha, \\
    &0
    && \text{otherwise,}
  \end{aligned} \right. 
\end{equation*}
where $\alpha > 0$ is chosen such that $\int_\R \widehat \psi = 1$. The function $\widehat \psi$ is constructed in \cite[Table 9.1]{Wendland05} to have compact support and $C^2$-regularity while being non-negative and positive definite \cite[Thm.~9.13]{Wendland05}. Then, by Bochner's Theorem, we also have that $\psi \geq 0$ is positive definite. By basic properties of the Fourier transform we find that $0 \leq \widehat \psi \leq 1$, $0 \leq \psi \leq 1$, $\widehat \psi (1) = \psi (1) = 1$, and $\psi, \widehat \psi \in X_{3,3} \subset X$. For the scaling in $\varepsilon$, we set
\begin{equation*}
  \psi_\varepsilon (x) = \frac1\e \psi \Bigl(\frac x\varepsilon\Bigr),
  \quad \text{and thus} \quad
  \widehat{ \psi_\varepsilon } (\omega) = \widehat \psi ( \varepsilon \omega).
\end{equation*}
We note that $\psi_\varepsilon \weakto \delta_0$ narrowly and $\widehat{ \psi_\varepsilon } \uparrow 1$ pointwise as $\varepsilon \to 0$. 

To check that $\nu_\varepsilon = \tilde \nu_\varepsilon \llcorner[0,\infty)$ is admissible, i.e.,~$\nu_\varepsilon \in \mathcal A$, we observe from $\psi_\varepsilon \geq 0$ having unit mass that the bounds on $\nu_\varepsilon^\pm$ are satisfied. We show that $T \nu_\varepsilon \in L^2 (\R)$ later in this proof.

Now we prove the vague convergence of $\nu_\varepsilon$ to $\nu$. Let $M > 0$, and let $\varphi \in C_b ([0, \infty))$ be an arbitrary function with support in $[0, M]$. Since $\psi_\varepsilon$ is even, we rewrite
\begin{equation} \label{for:app:Step:1:1}
  \int_0^M \varphi \nu_\varepsilon
  = \int_0^M \varphi \tilde \nu_\varepsilon
  = \int_0^M \varphi (\psi_\varepsilon \ast \nu)
  = \int_0^{M+1} (\psi_\varepsilon \ast \varphi) \nu + \int_{M+1}^\infty (\psi_\varepsilon \ast \varphi) \nu.
\end{equation}
Since $\psi_\varepsilon \ast \varphi$ converges to $\varphi$ uniformly in $\R$ as $\varepsilon \to 0$, the first integral in the right-hand side of \eqref{for:app:Step:1:1} converges to $\int \varphi \nu$. The second term in the right-hand side is bounded by Lemma \ref{lem:g:nu:bounds}, since $\psi_\e\ast \varphi \in X$, and moreover small in $\varepsilon$ since $\varphi$ is supported in $[0, M]$.

We observe that the second term of the energy $F (\nu_\varepsilon)$ converges to the corresponding term of the energy $F (\nu)$, i.e.,
\begin{equation*}
  \int_0^\infty g \nu_\varepsilon 
  \to \int_0^\infty g \nu
  \quad \text{as} \quad \varepsilon \to 0,
\end{equation*}
by Lemma \ref{lem:g:nu:bounds} and $g \in C_b ([0,\infty))$. It remains to show that $\nu_\varepsilon \in L^2 (\R)$ and that the first term of the energy $F (\nu_\varepsilon)$ satisfies the bound $\limsup_{\varepsilon \to 0} \| T \nu_\varepsilon \|_2 \leq \| T \nu \|_2$. Since $\widehat V > 0$, from Lemma \ref{lem:T:app}.\eqref{lem:T:app:T2:equals:Vast} we deduce that $\widehat \nu \in L^2_{\text{loc}} (\R; \C)$. Then, since $\widehat{ \psi_\varepsilon } \in X_{2,2}$, Lemma \ref{lem:F}.\eqref{lem:F:iv} implies that
\[
  \mathcal{F}\tilde \nu_\varepsilon
  = \widehat{ \psi_\varepsilon \ast \nu }
  = \widehat{ \psi_\varepsilon } \widehat{\nu}.
\]
We conclude from $\widehat{ \psi_\varepsilon }$ having bounded support and $\widehat \nu \in L^2_{\text{loc}} (\R; \C)$ that $\mathcal F \tilde \nu_\varepsilon \in L^2 (\R; \C)$, and thus $\tilde \nu_\varepsilon, \nu_\varepsilon \in L^2 (\R)$. 

Finally, we show that $\limsup_{\varepsilon \to 0} \| T \nu_\varepsilon \|_2 \leq \| T \nu \|_2$. To this aim, we split
\begin{equation} \label{for:app:Step:1:2}
  \| T \nu_\varepsilon \|_2 
  \leq \| T \tilde \nu_\varepsilon \|_2 + \big\| T \big({\tilde\nu_\varepsilon} \llcorner(-\infty, 0) \big) \big\|_2.
\end{equation}
We obtain from $\widehat{ \psi_\varepsilon } \leq 1$ and Lemma \ref{lem:T:app}.\eqref{lem:T:app:T2:equals:Vast} that
\begin{equation*}
  \| T \tilde \nu_\varepsilon \|_2^2 
  = \int_\R \widehat V \big| \widehat{ \tilde \nu_\varepsilon } \big|^2
  = \int_\R \widehat{ \psi_\varepsilon }^2 \widehat V | \widehat{ \nu } |^2
  \leq \int_\R \widehat V | \widehat{ \nu } |^2
  = \| T \nu \|_2^2.
\end{equation*} 
It remains to show that the second term in the right-hand side of \eqref{for:app:Step:1:2} is small in $\varepsilon$. Since $\hat \psi'' \in L^1(\R)$, $x^2\psi(x)\in C_b(\R)$, and therefore $\psi_\e(x) \leq C\e x^{-2}$. Since $\supp \nu \subset [c, +\infty)$ with $c > 0$ and $\sup_{x \geq 0} \nu([x,x+1]) < \infty$, it follows that for $x<0$
\[
|{\tilde \nu_\varepsilon}(x)|
\leq C\e \int_c^\infty \frac1{(x-z)^2} \,\nu(dz)
\leq C\e \sum_{k = 0}^\infty \frac1{(-x+ k+c)^2}.
\]
Hence, $\| {\tilde\nu_\varepsilon} \|_{L^2 (-\infty, 0)} \to 0$, and thus the second term in \eqref{for:app:Step:1:2} is small in $\varepsilon$.

\bibliographystyle{alpha} 
\bibliography{BLrefs}

\end{document}